\documentclass[11pt]{amsart}
\usepackage{amssymb,amsmath,amsfonts,amscd,euscript}
\usepackage[matrix,arrow,curve]{xy}
\usepackage{hyperref}

\newcommand{\nc}{\newcommand}

\numberwithin{equation}{section}
\newtheorem{thm}{Theorem}[section]
\newtheorem{prop}[thm]{Proposition}
\newtheorem{lem}[thm]{Lemma}
\newtheorem{cor}[thm]{Corollary}
\theoremstyle{remark}
\newtheorem{rem}[thm]{Remark}

\newtheorem{dfn}[thm]{Definition}

\nc{\fB}{\mathfrak{B}}
\nc{\gl}{\mathfrak{gl}}
\nc{\GL}{\mathfrak{GL}}
\nc{\g}{\mathfrak{g}}
\nc{\gh}{\widehat\g}
\nc{\h}{\mathfrak{h}}
\nc{\wfh}{\widehat{\mathfrak{h}}}
\nc{\la}{\lambda}
\nc{\al}{\alpha }
\nc{\be}{\beta }
\nc{\ve}{\varepsilon }
\nc{\om}{\omega }

\nc{\ta}{\theta}
\nc{\ch}{{\mathop {\rm ch}}}
\nc{\Tr}{{\mathop {\rm Tr}\,}}
\nc{\Id}{{\mathop {\rm Id}}}
\nc{\ad}{{\mathop {\rm ad}}}
\nc{\bra}{\langle}
\nc{\ket}{\rangle}
\nc{\bi}{{\bf i}}
\nc{\pa}{\partial}
\nc{\ld}{\ldots}
\nc{\cd}{\cdots}
\nc{\hk}{\hookrightarrow}
\nc{\T}{\otimes}
\nc{\gr}{\mathrm{gr}}
\nc{\ov}{\overline}

\nc{\cO}{\mathcal O}
\nc{\msl}{\mathfrak{sl}}
\nc{\mgl}{\mathfrak{gl}}
\nc{\U}{\mathrm U}
\nc{\V}{\EuScript V}
\nc{\cL}{\mathcal{L}}
\nc{\Res}{\mathrm{Res\ }}

\newcommand{\bC}{{\mathbb C}}

\newcommand{\bZ}{{\mathbb Z}}

\newcommand{\bP}{{\mathbb P}}
\newcommand{\bR}{{\mathbb R}}

\newcommand{\fh}{{\mathfrak h}}
\newcommand{\wh}{\widehat{\mathfrak h}}

\newcommand{\fg}{{\mathfrak g}}

\newcommand{\fb}{{\mathfrak b}}

\newcommand{\fn}{{\mathfrak n}}

\newcommand{\E}{\EuScript{E}}

\nc{\I}{\mathfrak I}
\nc{\bfI}{\mathbf I}
\nc{\Q}{\mathfrak Q}
\nc{\W}{\mathbb W}
\nc{\bU}{\mathbb U}

\nc{\Gm}{\mathbb{G}_{m}}
\nc{\bA}{\mathbb A}

\begin{document}

\title[Non-symmetric Macdonald polynomials at infinity]
{Representation theoretic realization of non-symmetric Macdonald polynomials at infinity}

\author{Evgeny Feigin}
\address{Evgeny Feigin:\newline
Department of Mathematics,\newline
National Research University Higher School of Economics,\newline
Usacheva str. 6, 119048, Moscow, Russia,\newline
{\it and }\newline
Skolkovo Institute of Science and Technology, Skolkovo Innovation Center, Building 3,
Moscow 143026, Russia
}
\email{evgfeig@gmail.com}

\author{Syu Kato}
\address{Syu Kato:\newline
Department of Mathematics, Kyoto University,
Oiwake Kita-Shirakawa Sakyo Kyoto 606--8502 JAPAN}
\email{E-mail:syuchan@math.kyoto-u.ac.jp}

\author{Ievgen Makedonskyi}
\address{Ievgen Makedonskyi:\newline
Max Planck Institute for Mathematics, Vivatgasse 7, 53111, Bonn, Germany
\newline
{\it and} \newline
Department of Mathematics,\newline
National Research University Higher School of Economics,\newline
Usacheva str. 6, 119048, Moscow, Russia
}
\email{makedonskii\_e@mail.ru}

\begin{abstract}
We study the nonsymmetric Macdonald polynomials specialized at infinity from various points of view.
First, we define a family of modules of the Iwahori algebra whose characters are equal to the
nonsymmetric Macdonald polynomials specialized at infinity. Second, we show that these modules are
isomorphic to the dual spaces of sections of certain sheaves on the semi-infinite Schubert varieties.
Third, we prove that the global versions of these modules are homologically dual to the level one
affine Demazure modules.
\end{abstract}

\maketitle

\section{Introduction}
Nonsymmetric Macdonald polynomials $E_\la(x,q,t)$ form a remarkable class of special functions (see
\cite{O,M3,Ch1,Ch2}).
They depend on a weight of a simple Lie algebra $\fg$ and variables $x=(x_1,\dots,x_n)$, $q$ and $t$.
Each $E_\la(x,q,t)$ is a polynomial in $x$-variables with coefficients being rational functions in $q$ and $t$.
The importance of the nonsymmetric Macdonald polynomials comes from numerous applications in combinatorics,
algebraic geometry and
representation theory. In particular, it has been shown in \cite{S,I} that the characters of the affine level
one
Demazure modules for the corresponding affine Kac-Moody Lie algebra are equal to the $t=0$ specializations
$E_\la(x,q,0)$.

It has been demonstrated recently that the $t=\infty$ specialization of the nonsymmetric Macdonald polynomials
is very meaningful as well (see \cite{CO1,CO2,OS,Kat,FeMa3,FeMa4,NS,NNS}). The study of the "opposite" $t=\infty$
specialization
has lead to various discoveries of representation theoretic, combinatorial and geometric nature. However, all
the
representation theoretic descriptions of $E_\la(x,q,\infty)$ obtained so far are dealing only with the nonsymmetric
Macdonald polynomials
corresponding to the anti-dominant weight $\la$ (recall that the Sanderson and the Ion theorems work for
arbitrary $\la$). The goal of this
paper is to fill this gap and to present the representation theoretic realization of $E_\la(x,q,\infty)$ for all weights.

Our starting point is a result from \cite{Kat} stating that there exists a geometric realization of all the
nonsymmetric Macdonald
polynomials at $t=\infty$. More precisely, it has been proved that for any dominant weight $\la$ and an element
$w\in W$
there exists a sheaf $\E_{w}(\la)$ on the semi-infinite Schubert variety $\Q(w)$ (see e.g. \cite{BF1,BF2,Kat,KNS}) such that
the character of the dual space of sections of $\E_{w}(\la)$ is equal (up to a simple factor) to the nonsymmetric
Macdonald polynomial
$E_{-w\la}(x^{-1},q^{-1},\infty)$ (see Section \ref{SoSIS} for more details). Moreover, this space is naturally
endowed with the structure
of a cyclic module over the Iwahori algebra. Our first main result is an explicit description of the
corresponding module of
the Iwahori. Namely, we put forward the following definition:
\begin{dfn}
Let $\la_-$ be an anti-dominant weight and let $\sigma$ be an element of the Weyl group of $\fg$.
The module $U_{\sigma(\la_-)}$ is the cyclic Iwahori module with cyclic vector $u_{\sigma(\la_-)}$ of $\fh$
weight $\sigma(\la_-)$
subject to the relations:
\begin{gather*}
\fh\T z\bC[z] u_{\sigma(\la_-)}=0,\\
\widehat{\sigma} (f_{-\al}\T z) u_{\sigma(\la_-)}=0,\ \al\in\Delta_+,\\
 (f_{\sigma(\al)}\T z)^{-\bra \la_-,\al^\vee\ket+1} u_{\sigma(\la_-)}=0,\ \al\in\Delta_+,
 \sigma\al\in\Delta_-,\\
(e_{\sigma(\al)}\T 1)^{-\bra \la_-,\al^\vee\ket} u_{\sigma(\la_-)}=0,\ \al\in\Delta_+, \sigma\al\in\Delta_+,
\end{gather*}
\end{dfn}
\noindent where the definition of the $\widehat\sigma$-action is given in \S \ref{curalg}.
The global version $\bU_{\sigma(\la_-)}$ is defined by the same set of relations with the first line omitted.
We prove the following theorem:
\begin{thm}\label{Th1}
For an anti-dominant weight $\la_-$ and $\sigma\in W$ one has
\[E_{\sigma(\la_-)}(x,q^{-1},\infty)=w_0\ch \, U_{w_0\sigma(\la_-)}.\]
\end{thm}

The above $U$-modules also give the spaces of sections of the sheaves $\E_{w}(\la)$ on a Schubert manifold $\Q ( w )$. More precisely, we prove
the following theorem.
\begin{thm}\label{Th2}
For a dominant weight $\la$ and $w\in W$ one has an isomorphism of the Iwahori modules
\[
H^0(\Q(w),\E_{w}(\la))^*\simeq	{\mathbb U}_{w(\la)}.
\]
\end{thm}

For an antidominant weight $\mu$ we consider a series $(q)_\mu^{-1}\in \bC [\![q]\!]$ (see section
\ref{curalg}).
In view of \cite[Corollary 6.10]{Kat}, Theorem \ref{Th2} implies
\begin{cor}\label{cor}
For an anti-dominant weight $\la_-$ and $\sigma\in W$ one has
\[(q)^{-1}_{(\la_-)_\sigma}\cdot E_{\sigma(\la_-)}(x,q^{-1},\infty)= w_0\ch \, \mathbb U _{w_0\sigma(\la_-)},\]
where $(\la_-)_\sigma$ is defined by \eqref{la-sigma}.
\end{cor}

Our third  theorem describes the categorical nature of the global $U$-modules. Let ${\mathfrak B}$ be the
category
of the Iwahori modules (see section \ref{cat} for the precise definitions). Let $D_\mu$ be the level one affine
Demazure module whose cyclic vector has weight $\mu$. Thanks to \cite{S,I}, the character of
$D_\mu$
is given by the $t=0$ specialization of the nonsymmetric Macdonald polynomial $E_\mu$ whenever $\g$ is of type $ADE$. We note that $D_\mu$, as well as $U_\la$, are elements of ${\mathfrak B}$. We prove that the global $\U$-modules are "dual" in the categorical
sense to the Demazure modules. More precisely, the following theorem holds.

\begin{thm}\label{Th3}
Assume that $\g$ is of type $ADE_{6}E_{7}$. We have:
\[
\mathrm{Ext}^{i}_{\mathfrak B} ( \mathbb U_{- \la}, D_{\mu}^{*} ) \cong
\begin{cases} \bC & (i=0,\la=\mu)\\\{0\} & (\text{otherwise}).\end{cases}
\]
\end{thm}
We conjecture that the theorem holds also for type $E_8$ as well.

Our paper is organized as follows. In Section \ref{Prelim} we collect main definitions we use in the
main body of the paper. In Section \ref{Umod} we study the representation theory
of the local and global $U$-modules. In Section \ref{NM} the link between the representation theoretical
properties
of the modules $U_\mu$ and the combinatorics of the nonsymmetric Macdonald polynomials is established;
in particular, we prove Theorem \ref{Th1}.
Section \ref{SoSIS} contains the study of the geometry of the semi-infinite Schubert varieties; in particular,
we prove Theorem \ref{Th2}. Finally, Section \ref{cat} is devoted to the study of the categorical properties of
the
modules $\bU_{w\la}$ and to the proof of Theorem \ref{Th3}. A combinatorial consequence of Theorem \ref{Th3}
is discussed in the Appendix.

\section{Preliminaries}\label{Prelim}

For a $\bZ$-graded vector space $V = \bigoplus_{m \in \bZ} V_{m}$, we set
$$\mathsf{gdim} \, V := \sum_{m \in \bZ} q^{m} \dim \, V_{m},$$
that is a priori a formal sum. We also define $V^{*} := \bigoplus_{m \in \bZ} V_{-m}^{*}$, where its degree
$m$-part is understood to be $V_{-m}^{*}$.

\subsection{Finite dimensional objects}
Let $\fg$ be a simple Lie algebra of rank $n$ with the Cartan decomposition $\fg=\fn_+\oplus\fh\oplus\fn_-$. Let
$\Delta=\Delta_+\sqcup \Delta_-\subset\fh^*$ be the set of roots and let $Q$ be the root lattice spanned
by $\Delta$. We set $I := \{1,2,\ldots,n\}$. We denote by $\{ \al_i \}_{i \in I}$ the set of simple roots and by
$\{ \om_i\}_{i \in I}$ the set of fundamental weights. For a root $\al\in\Delta$, we denote by $\al^\vee\in\fh$ the corresponding coroot.
For the standard pairing $\bra\cdot,\cdot\ket:\fh^*\times\fh\to\bC$ one has $\bra \om_i,\al_j^\vee \ket
=\delta_{i,j}$.

For each $\al\in\Delta_+$, we denote by $e_\al\in\fn_+$ the corresponding Chevalley generator of $\g$. Similarly, for
$\al\in\Delta_-$, we denote by $f_\al$ the Chevalley generator of weight $\al$ in $\fn_-$.
Let $P=\bigoplus_{i=1}^n\bZ\om_i$ be the weight lattice with the dominant cone
$P_+=\sum_{i=1}^n \bZ_{\ge 0} \om_i$. We set $P_- := - P_+$. For $\la\in P_+$, we denote the corresponding
irreducible finite-dimensional highest weight $\fg$-module by $V(\la)$. Let $v\in V(\la)$ be a non-zero highest weight vector.
Then $\fn_+ v=0$ and the defining relations of $V(\la)$ as $\fn_-$ module are of the form
\[
f_{-\al}^{\bra\la,\al^\vee\ket+1}v=0, \ \al\in\Delta_+.
\]

Let $\fg^\vee$ be the simple Lie algebra defined by the dual Kac-Moody data of $\fg$.

Finally, we denote by $W$ the finite Weyl group of $\fg$. For a root $\al$, the corresponding reflection is denoted by $s_\al\in W$. For each $i \in I$, we set $s_{i} := s_{\al_{i}}$. We sometimes identify $s_\al$ with $s_{\al^\vee}$ as the Weyl groups of $\fg$ and $\fg^{\vee}$ coincide.

\subsection{Current algebras} \label{curalg}
Let $\fg[z]=\fg\T \bC[z]$ be the current algebra. We have a grading on $\fg [z]$ by setting $\deg \, a \otimes z^m = m$ for each $a \in \g \setminus \{ 0 \}$ and $m \ge 0$. For $\la \in P_+$, we define the local Weyl module
$W(\la)$ of $\fg[z]$ as the cyclic $\fg[z]$-module with a cyclic vector $w$ of $\fh$-weight $\la$ subject to the relations
$\fh\T z\bC[z] w=0$, $\fn_+\T\bC[z] w=0$ and
\[
(f_{-\al}\T 1)^{\bra \la,\al^\vee\ket+1} w=0,\ \al\in\Delta_+.
\]
We define the global Weyl module $\W(\la)$ of $\fg [z]$ by omitting the condition $\fh\T z\bC[z] w=0$.
The characters of the local and global Weyl modules differ by a simple factor. Namely, let
us consider the subspace $A(\la)$ of weight $\la$ vectors in $\W(\la)$. Being a quotient of ${\rm U}(\fh\T\bC[z])$ by a homogeneous ideal, the vector space $A(\la)$ carries a structure of a graded commutative algebra whose grading is induced by the grading of $\g [z]$. We set $r_i= \bra \la,\al_i^\vee\ket$ for each $i \in I$. Then, the following holds:
\begin{itemize}
\item The algebra $A(\la)$ is isomorphic to the polynomial algebra in variables
$\{ x_{i,a} \}_{i \in I, 1\le a\le r_i}$ of degree one, symmetric in each group of variables
$x_{i,1},\dots,x_{i,r_i}$;
\item The algebra $A(\la)$ acts freely on the global Weyl module $\W(\la)$. The action commutes with the action of $\fg[z]$;
\item One has $\W(\la)/\mathfrak{m}\W(\la)\simeq W(\la)$, where $\mathfrak{m}$ is the ideal of
$A_{\lambda_+}$ consisting of polynomials without a constant term.
\end{itemize}

Let $\I = \fn^+\oplus\fh \oplus\fg\T z\bC[z]\subset\fg\T\bC[z]$ be the Iwahori subalgebra.
\begin{rem}\label{rem:cartan}
Note that the Cartan subalgebra $\fh$ is included in $\I$. In \cite{FeMa3,FMO} the authors used
the algebra $\fn^{af}= \fn^+\oplus\fg\T z\bC[z]\subset\fg\T\bC[z]$ instead of $\I$.
The only difference is that the Cartan part $\fh$ is missing in $\fn^{af}$, i.e.
$\I=\fn^{af}\oplus\fh$.
\end{rem}
An $\I$ module $M$ is called graded, if $M=\bigoplus_{j\in \bZ} M_j$ such that each $M_j$ is $\fh$ semi-simple
(i.e. each $M_j$ is the sum of $\fh$ weight spaces) and $(x\T z^i) M_j\subset M_{i+j}$ for all $x\in\fg$,
$i\ge 0$. We define the character of $M$ as the formal linear combination
\[
\ch \, M=\sum_{j\in\bZ} q^j\ch \, M_j,
\]
there $\ch \, M_j$ is the $\fh$-module character. In what follows we always consider the modules $M$ whose
$\fh$-weights belong to $P$. We say that $\ch \, M$ is well-defined whenever we have $\ch \, M_j \in \bZ [P]$
for each $j \in \bZ$. We set $x_i=e^{\om_i}$. Then we have
$\ch \, M \in \bZ [x_1^{\pm},\dots,x_n^\pm][\![q,q^{-1}]\!]$ when $\ch \, M$ is well-defined.
If $M$ is cyclic with cyclic vector $v$, then we assume that
$v\in M_0$ unless stated otherwise.

One concludes that
$$\ch \, \W(\la) = (q)_{\la}^{-1} \cdot \ch \, W(\la),$$
where we have
$$(q)_{\la}=\prod_{i=1}^n\prod_{j=1}^{r_i} (1-q^j) = ( \mathsf{gdim} \, A(\la) )^{-1} \in \bZ [\![q]\!].$$
Extending this, we set $(q)_{w \la} := (q)_{\la}$ for each $w \in W$.

For each $\mu \in P$, let us denote by $E_{\mu}(x,q,t) \in \bC [P](q,t)$ the non-symmetric Macdonald polynomial in the sense of Cherednik \cite{Ch1}. In particular, the character of ${\mathbb W}(\la)$ for $\la \in P_{+}$ is given by
$$\mathrm{ch} \, {\mathbb W}(\la) = (q)_\la^{-1} \cdot E_{w_0\la}(x,q,0),$$
(cf. \cite{S,I,FL}).

Recall the $W$-action on the set
$\{e_\al\}_{\al\in\Delta_+}\cup\{f_{-\al}\T z\}_{\al\in\Delta_+}$ following \cite{FeMa3}: for an element
$\sigma\in W$
and $\al\in\Delta_+$ we set
\[
\widehat{\sigma} e_\al=\begin{cases}
e_{\sigma(\al)}, \qquad  \sigma(\al)\in\Delta_+,\\
f_{\sigma(\al)}\T z, \ \sigma(\al)\in\Delta_-,
\end{cases}
\widehat{\sigma} (f_{-\al}\T z)=\begin{cases}
e_{-\sigma(\al)}, \qquad  \sigma(\al)\in\Delta_-,\\
f_{-\sigma(\al)}\T z, \ \sigma(\al)\in\Delta_+.
\end{cases}
\]
We will also use the following notation for $\al\in\Delta_+$ and $r\ge 0$:
\begin{gather*}
e_{\widehat{\sigma}\al+r\delta}=\begin{cases}
e_{\sigma(\al)}\T z^r, &  \sigma(\al)\in\Delta_+,\\
f_{\sigma(\al)}\T z^{r+1}, & \sigma(\al)\in\Delta_-,
\end{cases},\\
e_{\widehat{\sigma}(-\al+\delta)+r\delta}=\begin{cases}
e_{-\sigma(\al)}\T z^r, &  \sigma(\al)\in\Delta_-,\\
f_{-\sigma(\al)}\T z^{r+1}, & \sigma(\al)\in\Delta_+.
\end{cases}
\end{gather*}

\subsection{Affine algebras}\label{Aa}
The affine Weyl group $W^{a}$ and the extended affine Weyl group $W^e$ attached to $\fg^\vee$ fits into the following commutative diagram:
$$
\xymatrix{
1 \ar[r] & Q \ar[r]\ar@{^{(}->}[d] & W^a \ar[r] \ar@{^{(}->}[d] & W \ar[r] \ar@{=}[d] & 1\\
1 \ar[r] & P \ar[r]^t & W^e \ar[r] & W \ar[r] & 1
},
$$
where the action of $W$ on $Q$ and $P$ are the standard ones. In particular, every element $w\in W^e$ can be uniquely
written as $w=t({\rm wt}(w)){\rm dir}(w)$, where ${\rm wt}(w) \in P$ and ${\rm dir}(w) \in W$. In what follows, we sometimes write $t_\mu$ for the image of $\mu \in P$ through the map $t$.

Let $\widehat\fg^\vee$ be the untwisted affine Kac-Moody algebra corresponding to $\fg^\vee$. The real roots of this dual affine Kac-Moody
algebra are of the form $\beta=\bar\beta + \delta^\vee\deg\beta$, where $\bar\beta\in\Delta^\vee$ and
$\delta^\vee$ is the primitive null root. We sometimes call the roots of $\widehat\fg^\vee$ the affine coroots.
Let us denote by $\widetilde\Delta^{a}_+$ the set of positive affine coroots, and denote by $\Delta^{a}_+$ the set of positive affine roots. For an element $w\in W^e$, the length $\ell(w)$
is defined as
$$\ell ( w ) := \# \{\beta \in \widetilde\Delta^{a}_+ \mid w ( \beta ) \not\in \widetilde\Delta^{a}_+ \}.$$
The set of length zero elements is denoted by $\Pi$. One has a semi-direct product decomposition $W^e=\Pi\rtimes W^a$.

It is standard that the positive level affine action of $W^a$ on $P \otimes_{\bZ} \bR$ identifies $P \otimes_{\bZ} \bR$ with the (closure of the) union of the $W^a$-translations of the fundamental region called the fundamental alcove \cite{Lus}. In the same fashion, we regard $\Pi \times ( P \otimes_{\bZ} \bR )$ as
the (closure of the) union of the $W^a$-translations of the fundamental alcove. For each $\pi \in \Pi$, we refer to $\pi \times ( P \otimes_{\bZ} \bR )$ as a sheet.
The set of alcoves contained in $\pi \times ( P \otimes_{\bZ} \bR )$ is in bijection with $\pi W^a$.

Let $u(\la)$ be the minimal length element in the coset $t_\la W$. We have $t _\la = u(\la)$ if and only if $\la \in P_-$
(see e.g. \cite[(2.4.5)]{M3}). Let $v(\la)\in W$ be the shortest element such that $v(\la)\la\in P_-$.
For $\la_-= v(\la)\la$ one has
\[
t_\la=u(\la)v(\la)\ \text{ and }\ t_{\la_-}=v(\la)u(\la).
\]
One also has $\ell(t_{\la_-})=\ell(v ( \la ))+\ell(u ( \la ) )$.

\subsection{Quantum Bruhat graphs}\label{qBg}

The Bruhat graph BG of $W$ (see e.g. \cite{BB}) is the directed graph whose set of vertices is identified with $W$ and we have an
arrow $w \to w s_\al$
for $w\in W$ and $\al \in \Delta_+$ if and only if $\ell ( w s_\al )=\ell ( w )+1$.
The quantum Bruhat graph QBG of $W$ (see e.g. \cite{BFP,LNSSS1}) is an enhancement of BG obtained by adding
a ``quantum" arrow $w \to w s_\al$ for each $w\in W$ and $\al \in \Delta_+$ so that
$$\ell ( ws_\al ) = \ell ( w ) - \sum_{\gamma \in \Delta_+} \bra \gamma, \al^{\vee} \ket + 1.$$
Note that QBG is obtained as the image of the covering relation graph of the periodic $W^a$-graph through the projection $W^a = W \ltimes Q \rightarrow W$ (see \cite{Lus, LNSSS1}). We denote the projection obtained as its enhancement $W^e = W \ltimes P \rightarrow W$ by ${\rm dir}$.

Assume that we are given an element $z_0\in W^e$ and a sequence of affine coroots
$\beta_1,\dots,\beta_l$. A path $p_J$ ($=:p$) corresponding to a set
$$J=\{1\le j_1<\dots <j_r\le l\} \subset 2^{[1,l]}$$
is a sequence $p_J=(z_0,z_1,\dots,z_r)$, where $z_{k+1}=z_ks_{\beta_{j_{k+1}}}$. Here we refer the last element $z_r$ as the end of the path $p$, and denote it by ${\rm end}(p)$ ($= {\rm end}(p_J) =z_r$). We say that $p$ is a quantum alcove path if ${\rm dir}$ induces the following path in the quantum Bruhat graph:
\[
{\rm dir}(z_0) \stackrel{\bar\beta_{j_{1}}}{\longrightarrow} {\rm dir}(z_1) \stackrel{\bar\beta_{j_{2}}}{\longrightarrow}
\dots \stackrel{\bar\beta_{j_{r}}}{\longrightarrow}{\rm dir}(z_r).
\]
For an alcove path $p$ we denote by $qwt^*(p)$ the sum of all $\beta_{j_k}$ such that the edge
${\rm dir}(z_{k-1}) \stackrel{\bar\beta_{j_k}}{\longrightarrow} {\rm dir}(z_k)$ is quantum.

For $u\in W^e$ one denotes by ${\mathcal{QB}}(id; u)$ the set of quantum alcove paths
with $z_0=u$ and with $\beta$'s coming from a fixed reduced decomposition of $u$ (see \cite{OS}). Also one denotes
by $\overleftarrow{\mathcal{QB}}(id; u)$ the set of alcove paths which project to some path in the graph obtained by QBG by reverting all the edges (the reversed quantum Bruhat graph). Note that both sets ${\mathcal{QB}}(id; u)$ and $\overleftarrow{\mathcal{QB}}(id; u)$ depend
on the reduced expression of $u$.

\subsection{Graded homomorphisms}
By a graded abelian category $\mathfrak C$, we mean an abelian category $\mathfrak C$ equipped with an autoequivalence $M \mapsto M \left< n \right>$ for each $M \in \mathfrak C$ and $n \in \bZ$ so that we have a
functorial isomorphism
$$( M \left< n \right> ) \left< - n \right> \cong M.$$

In this setting, we define
$$\mathrm{hom}_{\mathfrak C} ( M, N ) := \bigoplus_{n \in \bZ} \mathrm{Hom}_{\mathfrak C} ( M \left< n \right>,
N ) \hskip 5mm M, N \in \mathfrak C.$$
We regard $\mathrm{hom}_{\mathfrak C} ( M, N )$ as a $\bZ$-graded vector space whose $n$-th graded piece is given by
$$\mathrm{hom}_{\mathfrak C} ( M, N )_n := \mathrm{Hom}_{\mathfrak C} ( M \left< n \right>, N ).$$

In case $\mathfrak C$ has enough projectives, then we also define
$$\mathrm{ext}^i_{\mathfrak C} ( M, N ) := \bigoplus_{n \in \bZ} \mathrm{Ext}^i_{\mathfrak C} ( M \left< n \right>, N ).$$

They have the usual long exact sequences associated to short exact sequences (degreewise). For a graded abelian category $\mathfrak C$, we denote by $[\mathfrak C]$ its Grothendieck group. The group $[\mathfrak C]$ naturally admits a $\bZ [q,q^{-1}]$-module structure by identifying the action of $\left< 1 \right>$ with $q$.

\section{U-modules}\label{Umod}

In this section, we introduce two families of $\I$-modules $\{ U_{\la} \}_{\la \in P}$ and $\{ \mathbb U_{\la} \}_{\la \in P}$ parametrized by $P$ that represents non-symmetric Macdonald polynomials and those divided by their norms. We refer to $U_{\la}$ as the local $U$-module and
to $\mathbb U_\la$ as the global $U$-module. The final results are presented in subsection \ref{NM}. The main ideas are to find a
surjection $\W_{\la} \to \mathbb U_{\la}$, and refine the decomposition procedure from \cite{FeMa3, FMO} to identify $\mathbb U_{\la}$ with
a module with prescribed characters and defining equations (Theorem \ref{local}, Theorem \ref{decompositionglobal},
and Corollary \ref{combmodel}).

\subsection{Definitions}
For an anti-dominant weight $\lambda_-$ and an element $\sigma\in W$ we define two modules
$U_{\sigma(\la_-)}$ and ${\mathbb U}_{\sigma(\la_-)}$ as follows:
$U_{\sigma(\la_-)}$ is the cyclic $\I$-module with cyclic vector $u_{\sigma(\la_-)}$ of $\fh$-weight
$\sigma(\la_-)$
subject to the relations:
\begin{gather*}
\fh\T z\bC[z] u_{\sigma(\la_-)}=0,\\
(e_{\widehat{\sigma}(-\alpha+\delta)+r \delta}) u_{\sigma(\la_-)}=0,\ \al\in\Delta_+,r \geq 0 \\
 (f_{\sigma(\al)}\T z)^{-\bra \la_-,\al^\vee\ket+1} u_{\sigma(\la_-)}=0,\ \al\in\Delta_+,
 \sigma\al\in\Delta_-,\\
(e_{\sigma(\al)}\T 1)^{-\bra \la_-,\al^\vee\ket} u_{\sigma(\la_-)}=0,\ \al\in\Delta_+, \sigma\al\in\Delta_+.
\end{gather*}
The definition of the $\I$-module ${\mathbb U}_{\sigma(\lambda_-)}$ differs from the definition of the
$U_{\sigma(\la_-)}$ by removing the first line relation: $\fh\T z\bC[z] u_{\sigma(\la_-)}=0$.

\begin{rem}
Let $\fg$ be of type ADE, and let $D_{\sigma(\la_-)}$ be the level one affine Demazure module whose cyclic vector has weight $\sigma(\la_-)$ (if we restrict the weight to $\fh$; see e.g. \cite{FL}). Then
$D_{\sigma(\la_-)}$ is the cyclic $\I$-module with the same set of relations as for $U_{\sigma(\la_-)}$ with the last two lines
replaced with
\begin{gather*}
(f_{\sigma(\al)}\T z)^{-\bra \la_-,\al^\vee\ket} u_{\sigma(\la_-)}=0,\ \al\in\Delta_+, \sigma\al\in\Delta_-,\\
(e_{\sigma(\al)}\T 1)^{-\bra \la_-,\al^\vee\ket+1} u_{\sigma(\la_-)}=0,\ \al\in\Delta_+, \sigma\al\in\Delta_+.
\end{gather*}
i.e. $+1$ got moved down from the third relation to the fourth relation.
\end{rem}

\begin{rem}
Let $W_{\sigma(\la_-)}$ be the generalized Weyl module (\cite{FeMa3,FMO,No}). Then
$W_{\sigma(\la_-)}$ is the cyclic $\I$-module with the same set of relations as for $U_{\sigma(\la_-)}$ with the last
two lines
replaced with
\begin{gather*}
(f_{\sigma(\al)}\T z)^{-\bra \la_-,\al^\vee\ket+1} u_{\sigma(\la_-)}=0,\ \al\in\Delta_+, \sigma\al\in\Delta_-,\\
(e_{\sigma(\al)}\T 1)^{-\bra \la_-,\al^\vee\ket+1} u_{\sigma(\la_-)}=0,\ \al\in\Delta_+, \sigma\al\in\Delta_+.
\end{gather*}
i.e. $+1$ is now present in both of the last two relations.
\end{rem}

In the below, we denote the cyclic vector $u_{\la} \in W _{\la}$ by $w_{\la}$ in order to distinguish it from the cyclic vector of $U_{\la}$ (where $\la = \sigma ( \la_- )$).

\begin{cor}
Under the above settings, we have natural surjections of $\I$-modules $W_{\sigma(\la_-)}\to D_{\sigma(\la_-)}$ and $W_{\sigma(\la_-)}\to
U_{\sigma(\la_-)}$.
\end{cor}

\begin{proof}
Clear from the comparison of the defining relations.
\end{proof}

\begin{rem}\label{WDU-rel}
{\bf 1)} For $\la_- \in P_-$, the surjection $W_{\la_-}\to D_{\la_-}$ is an
isomorphism; {\bf 2)} For $\la \in P_+$, the surjection $W_{\la}\to U_{\la}$ is an
isomorphism.
\end{rem}

\begin{thm}[\cite{FeMa3,FMO,Kat}]
For $\la \in P_{+}$, the character of ${\mathbb W}_\la$ is given by
$$\mathrm{ch} \, {\mathbb W}_{\la} = (q)_\la^{-1} \cdot w_0 E_{w_0\la}(x,q^{-1},\infty).$$
\end{thm}

\subsection{From local to global}

Now assume that $\la_- \in P_-$ and $\sigma \in W$ satisfy $\bra \la_-,\al_i^\vee\ket<0$ for each $i \in I$ such that
$\sigma\al_i\in\Delta_+$. We define
\begin{equation}\label{la-sigma}
(\la_-)_\sigma=\la_-+\sum_{j: \sigma\al_j>0}\om_j\in P_-.
\end{equation}

\begin{rem}\label{><}
In \cite{Kat}, the weight $(\la_+)_w$ is defined for each $\lambda_+ \in P_+$ and $w\in W$ as $\la_+-\sum_{j: w\al_j<0}\om_j$. Our notation $(\la_-)_\sigma$ (that follows \cite{FeMa3,FMO}) is different from \cite{Kat}, and the relation between two notations are: $w\to \sigma w_0$, and $\la_+\to w_0\la_-$. Note that we have
$\sigma\la_-= w\la_+$, and one has
\[
w_0(\la_-+\sum_{j: (w w_0)\al_j>0}\om_j)=w_0\la_- + \sum_{w (w_0\al_j)>0} w_0\om_j = \la_+ - \sum_{w
\al_{w_0j}<0} \om_{w_0j},
\]
where we define the number $w_0j$ by $\om_{w_0j}=-w_0\om_j$.
This implies that $w_0(\la_-)_\sigma=(\la_+)_w$.
\end{rem}

\begin{lem}\label{UW}
There exists surjective homomorphisms of $\fn^{af}$-modules $U_{\sigma(\la_-)}\to W_{\sigma((\la_-)_\sigma)}$
and $($its global version$)$
${\mathbb U}_{\sigma(\la_-)}\to {\mathbb W}_{\sigma((\la_-)_\sigma)}$ induced by sending $u_{\sigma(\la_-)}$ to $w_{\sigma((\la_-)_\sigma)}$.
\end{lem}
\begin{proof}
The proofs of the both assertions are by checking the relations of the LHS in the RHS. To this end, there are essentially no difference in the proofs of the both cases, and we only exhibit the global case.

We show that all the defining relations of ${\mathbb U}_{\sigma(\la_-)}$ imposed on $u_{\sigma(\la_-)}$ hold for $w_{\sigma((\la_-)_\sigma)} \in {\mathbb
W}_{\sigma((\la_-)_\sigma)}$.

Let $\al\in\Delta_+\cap\sigma^{-1} \Delta_+$. We need to verify that
\begin{equation}\label{UtoW}
(e_{\sigma(\al)}\T 1)^{-\bra \la_-,\al^\vee\ket} w_{\sigma((\la_-)_\sigma)}=0
\end{equation}
in ${\mathbb W}_{\sigma((\la_-)_\sigma)}$, while we have
\[
(e_{\sigma(\al)}\T 1)^{-\bra (\la_-)_\sigma,\al^\vee\ket+1} w_{\sigma((\la_-)_\sigma)}=0
\]
in ${\mathbb W}_{\sigma((\la_-)_\sigma)}$ by definition.

Our assumption from the beginning of this subsection says that we have $\bra \la_-,\al_i^\vee\ket<0$ for each $i \in I$ so that $\sigma\al_i\in\Delta_+$. Since $\sigma\al\in \Delta_+$, there exists at least one simple root $\al_i$ so that $\al_i$ appears with positive multiplicity if we write $\al$ by a non-negative integer linear combination of simple roots, $\sigma\al_i\in\Delta_+$, and $\bra\la_-,\al^\vee_i\ket<0$. Hence, we have
\[
-\bra (\la_-)_\sigma,\al^\vee\ket+1\le -\bra \la_-,\al^\vee\ket
\]
and \eqref{UtoW} holds.

It remains to show that for $\al\in\Delta_+\cap\sigma^{-1} \Delta_-$, it holds that
\[
(f_{\sigma(\al)}\T z)^{-\bra \la_-,\al^\vee\ket+1} w_{\sigma(\la_-)}=0
\]
in ${\mathbb W}_{\sigma((\la_-)_\sigma)}$. This follows from the obvious inequality
$-\bra (\la_-)_\sigma,\al^\vee\ket\le -\bra \la_-,\al^\vee\ket$ for any $\al\in\Delta_+$.
\end{proof}

\begin{thm}[\cite{FMO}]\label{FMO-free}
The graded vector space of $\fh$-weight $\sigma(\la_-)$ vectors of $\W_{\sigma(\la_-)}$ affords a regular representation
of the algebra $A(w_0\la_-)$ described in section \ref{Prelim} through the action of
$$A ( w_0 \la _-) \longrightarrow \mathrm{End}_{\I} ( \W_{\sigma(\la_-)} )$$
on the cyclic vector of $\W_{\sigma(\la_-)}$. 
Moreover, $A(w_0\la_-)$ acts freely on $\W_{\sigma(\la_-)}$ and the quotient with respect to the augumentation ideal is isomorphic to the local generalized Weyl module $W_{\sigma(\la_-)}$. \hfill $\Box$
\end{thm}

\begin{cor}
In the same setting as in Theorem \ref{FMO-free}, we have
$$\ch \, \W_{\sigma(\la_-)}=(q)_{\la_-}^{-1} \cdot \ch \, W_{\sigma(\la_-)}.$$
\end{cor}

\begin{prop}\label{f-factor}
The character of the $\fh$-weight $\sigma(\la_-)$ vectors in ${\mathbb U}_{\sigma(\la_-)}$ is equal to
$(q)_{(\la_-)_\sigma}^{-1}$.
\end{prop}
\begin{proof}
Let $d ( \la_-)$ denote the characters of the $\fh$-weight $\sigma(\la_-)$ vectors in ${\mathbb U}_{\sigma(\la_-)}$. The combination of Theorem \ref{FMO-free} and Lemma \ref{UW} yields that the character of the weight $\sigma(\la_-)$ vectors in
${\mathbb U}_{\sigma(\la_-)}$ is greater than or equal to the character of the weight $\sigma((\la_-)_\sigma)$ vectors in $\W_{\sigma((\la_-)_\sigma)}$. Namely, we have
$$d ( \la_-) \le (q)_{(\la_-)_\sigma}^{-1} \hskip 3mm \text{inside} \hskip 3mm \mathbb Z [\![q]\!].$$

We prove that the algebra $A_{({\lambda_-})_\sigma}$ surjects onto the space of weight $\sigma(\la_-)$ vectors in
${\mathbb U}_{\sigma(\la_-)}$. Note that $A_{({\lambda_-})_\sigma}$ is a quotient of $U ( \fh [z] )$ if we regard it as automorphism on the space of weight $(\la_-)_\sigma$ vectors on $\W_{(\la_-)_\sigma}$. By the comparison with the $\mathfrak{sl} ( 2, \mathbb C [z] )$-calculation, we deduce that all the defining equations of $A_{({\lambda_-})_\sigma}$ are captured by the actions on the $\mathfrak{sl} ( 2 )$-triples (and their twists by $z^m$) on the cyclic vector.

For each $i \in I$, we have the relations in ${\mathbb U}_{\sigma(\la_-)}$
\begin{gather*}
(e_{\sigma(\al_i)}\T 1)^{-\bra\la_-,\al_i^\vee\ket}u_{\sigma(\la_-)}=(f_{-\sigma(\al_i)}\T z)u_{\sigma(\la_-)}=0
\text{ if } \sigma(\al_i)\in\Delta_+,\\
(f_{\sigma(\al_i)}\T z)^{-\bra\la_-,\al_i^\vee\ket+1}u_{\sigma(\la_-)}=(e_{-\sigma(\al_i)}\T
1)u_{\sigma(\la_-)}=0
\text{ if } \sigma(\al_i)\in\Delta_-.
\end{gather*}
These relations are enough to quotient out a tensor factor of $A_{({\lambda_-})_\sigma}$ corresponding to $i \in I$ inside
$U ( \mathbb C (\sigma \alpha_i^{\vee}) [z] ) \subset U ( \fh [z] )$. Since these defining relations hold for all $i\in I$,
we conclude that the weight $\sigma(\la_-)$-part of ${\mathbb U}_{\sigma(\la_-)}$ is a quotient of $A_{({\lambda_-})_\sigma}$.

This is the desired surjection, and we deduce
$$\ch \, A_{({\lambda_-})_\sigma} \le d ( \la_-) \le (q)_{(\la_-)_\sigma}^{-1}.$$
Now the equality
$$\ch \, A_{({\lambda_-})_\sigma}=(q)_{(\la_-)_\sigma}^{-1}$$
implies $d ( \la_-) = (q)_{(\la_-)_\sigma}^{-1}$, that completes the proof.
\end{proof}

\begin{cor}\label{U-free}
$\ch \, {\mathbb U}_{\sigma(\la_-)}\le (q)_{(\la_-)_\sigma}^{-1} \cdot \ch \, {U}_{\sigma(\la_-)}$.
In case the equality holds, the $A_{({\lambda_-})_\sigma}$-action on ${\mathbb U}_{\sigma(\la_-)}$ is free.
\end{cor}
\begin{proof}
We have the action of the algebra  $A_{({\lambda_-})_\sigma}$ on the module ${\mathbb U}_{\sigma(\la_-)}$
making ${\mathbb U}_{\sigma(\la_-)}$ an
$( U(\I), A_{({\lambda_-})_\sigma} )$-bimodule (${\mathbb U}_{\sigma(\la_-)}$ is the quotient of generalized Weyl module
$\W_{\sigma \lambda_-}$ and we thus have an action of $A_{\sigma \lambda_-}$ by \cite{FMO}, Proposition 3.8).
Let $\mathfrak{m}$ be the ideal of $A_{({\lambda_-})_\sigma}$ consisting of polynomials without
free term. Then we have:
\[{\mathbb U}_{\sigma(\la_-)}/{\mathbb U}_{\sigma(\la_-)} \cdot \mathfrak{m} \simeq {U}_{\sigma(\la_-)}.\]
Therefore we have:
\[\ch \, {\mathbb U}_{\sigma(\la_-)}\le \ch \, {U}_{\sigma(\la_-)} \cdot \ch \, A_{({\lambda_-})_\sigma}.\]
This completes the proof. Note that by the graded version of the Nakayama lemma the equality means that the
action of $A_{({\lambda_-})_\sigma}$ is free.
\end{proof}

\subsection{Generalized Weyl modules with characteristic}\label{gWc}

Let $\la_-,\mu\in P_-$ be such that $\la_- - \mu\in P_-$. We fix a reduced
decomposition
\begin{equation}\label{pi}
t_\mu=\pi s_{j_1}\dots s_{j_l},\ \pi\in\Pi,\ l=\ell(t_\mu)
\end{equation}
in the extended affine Weyl group and consider the affine coroots $\beta_1,\dots,\beta_l$ defined by
\[
\beta_l(\mu)=\al_{j_l}^\vee, \beta_{l-1}(\mu)=s_{j_l}\al_{j_{l-1}}^\vee,\dots, \beta_1(\mu)=s_{j_l}\dots
s_{j_2}\al_{j_1}^\vee
\]
(see \cite{OS}). 
In what follows, we may omit $\mu$ in the notation 
$\beta_j(\mu)$ (i.e. we may refer to $\beta_j(\mu)$ as $\beta_j$) if no confusion is possible.
We have the decomposition $\beta_j=\bar\beta_j+ ({\rm deg}\, \beta_j )\delta^\vee$, where
$\bar \beta_j\in \Delta^\vee$ and ${\rm deg} \, \beta_j \in \mathbb Z$. We note that  $\bar\beta_j$ is always a negative coroot and ${\rm deg}\,
\beta_j>0$ as (\ref{pi}) is a reduced expression.

For a positive root $\al$ and a number $m=1,\dots,l$ we define
\begin{equation}\label{l}
l_{\al,m}=-\bra\la_-,\al^\vee\ket - |\{j:\ \bar\beta_j=-\al^\vee, 1\le j\le m\}|.
\end{equation}

\begin{dfn}[\cite{FeMa3,FMO}]
The generalized Weyl module with characteristics $W_{\sigma(\la_-)}(m)$ is the $\fn^{af}$ module which is the quotient of $W_{\sigma(\la_-)}$ by the submodule generated by
\begin{equation}\label{charrel}
e_{{\widehat \sigma}(\al)}^{l_{\al,m}+1} w_{\sigma(\la_-)}, \ \al\in\Delta_+.
\end{equation}
(Recall that $w_{\sigma(\la_-)}$ is the cyclic vector of $W_{\sigma(\la_-)}$.) Similarly, we define the generalized global Weyl module with characteristics $\mathbb W_{\sigma(\la_-)}(m)$ as the $\fn^{af}$ module quotient of $\mathbb W_{\sigma(\la_-)}$ by the submodule generated by (\ref{charrel}) inside $\mathbb W_{\sigma(\la_-)}$.
\end{dfn}

\begin{rem}
In order to make $W_{\sigma(\la_-)}(m)$ into an $\I$-module, one has to specify the weight of the cyclic vector.
If the opposite is not stated explicitly, we assume that
the weight of the cyclic vector of $W_{\sigma(\la_-)}(m)$ is equal to $\sigma(\la_-)$, so that the natural
surjection map $W_{\sigma(\la_-)}\to W_{\sigma(\la_-)}(m)$ is an $\I$-module homomorphism.
\end{rem}

\begin{lem}
One has the natural chain of surjections of $\fn^{af}$-modules
\[
W_{\sigma(\la_-)}=W_{\sigma(\la_-)}(0)\to W_{\sigma(\la_-)}(1) \to\dots\to
W_{\sigma(\la_-)}(l)=W_{\sigma(\la_--\mu).}
\]
\end{lem}
\begin{proof}
For any positive root $\al$ the number of $i$ ($1\le i\le \ell(t_\mu)$) such that $\bar\beta_i=-\al^\vee$
is equal to $-\bra \mu,\al^\vee\ket$ by counting the effect of the conjugation action on $\Delta^a_+$. Hence, the assertion follows from \eqref{l} and the definition of the generalized Weyl modules with characteristics (see relations \eqref{charrel}).
\end{proof}

\subsection{The decomposition procedure}









Recall the notation (\ref{l}).
To simplify the notation, we denote $l_m:=l_{-\bar\beta_m,m-1}$.
The vector $e_{{\widehat \sigma}(-\bar\beta_m)}^{l_{m}} w _{\sigma ( \la_-)} \in W_{\sigma(\la_-)}(m-1)$ generates the
$\I$-submodule $W_{\sigma(\la_-)}' (m) \subset W_{\sigma(\la_-)}(m-1)$. We set
$$w _{\sigma ( \la_-), m} := e_{{\widehat \sigma}(-\bar\beta_m)}^{l_{m}} w _{\sigma ( \la_-)} \in W_{\sigma(\la_-)}(m-1).$$

We first prepare a lemma.
Let $x$ be an arbitrary element of the extended affine Weyl group. 
We consider two reduced decompositions
\begin{gather}
x=\pi s_{i_1}\dots s_{i_{r}}s_{i_{r+1}}\dots s_{i_{r+s}}s_{i_{r+s+1}}\dots s_{i_l},\label{x1}\\
x=\pi s'_{i_1}\dots s'_{i_{r}}s'_{i_{r+1}}\dots s'_{i_{r+s}}s'_{i_{r+s+1}}\dots  s'_{i_l}\label{x2}
\end{gather} such that
$s_{i_a}=s'_{i_a}$ for $a \notin \{r+1, \dots, r+s\}$ and
$s_{i_{r+1}}\dots s_{i_{r+s}}=s_{i_{r+1}}'\dots s_{i_{r+s}}'$
is a Coxeter relation. Let $\{\beta_i\}_{i=1}^l$, $\{\beta_i'\}_{i=1}^l$ be the corresponding sequences of $\beta$'s
constructed via reduced decompositions \eqref{x1},\eqref{x2}.

We take a reduced decomposition of the longest element $w_0=s_{j_1}\dots s_{j_l}$ of the finite Weyl group.
Define the sequence of roots $s_{j_l}\dots s_{j_2}(\alpha_{j_1})$, $s_{j_l}\dots s_{j_3}(\alpha_{j_2})$, $\dots, \alpha_{j_l}$. This sequence consists
of all positive roots each one time. An order on the set of positive roots obtained in such a way from some reduced
decomposition of the longest element is called a convex order \cite{P}. Note that for Lie algebras of rank $2$ there are exactly
two convex orders.
\begin{lem}\label{braidonbetas}
$(i)$ $\beta_a=\beta_a'$ for $a \notin \{r+1, \dots, r+s\}$;

$(ii)$ $\beta_{r+1}, \dots, \beta_{r+s}$ is the sequence of positive coroots of some rank-two Lie algebra in the convex order
of type $A_1\times A_1$ if $s=2$, $A_2$ if $s=3$, $C_2$ if $s=4$ and $G_2$ if $s=6$;

$(iii)$ $\beta_{r+b}=\beta_{r+s+1-b}'$ for $1 \leq b \leq s$, i.e. the Coxeter relation results in inverting the
convex order for some root subsystem $($of $\Delta)$ of rank $2$.
\end{lem}
\begin{proof}
The claim $(i)$ is obvious.

Let $\tau=s_{i_l} \dots s_{i_{r+s+1}}$. It is easy to see that the sequence
$$\alpha_{i_{r+s}}^\vee,s_{i_{r+s}}(\alpha^\vee_{i_{r+s-1}}), \dots, (s_{i_{r+s}}\dots s_{i_{r+2}})\alpha^\vee_{i_{r+1}}$$ is
the sequence (wtitten in a in convex order) of all positive coroots of rank two Lie algebra with root system spanned by
$\alpha_{i_{r+1}},\alpha_{i_{r+2}}$  and the Coxeter relation inverses the order of
these elements. Thus $\beta_{r+1}, \dots, \beta_{r+s}$ is the sequence of all positive coroots of rank two Lie algebra with
root system spanned by
$\tau(\alpha_{i_{r+1}}^\vee),\tau(\alpha_{i_{r+2}}^\vee)$ in convex order. This implies $(ii)$ and $(iii)$.
\end{proof}

Let  $\gamma_1, \gamma_2 \in \Delta$ be two linear independent roots.
We consider the root system $\Delta_2=\mathbb{Z}\langle \gamma_1, \gamma_2 \rangle\cap \Delta$,
$\Delta_{2+}=\mathbb{Z}\langle \gamma_1, \gamma_2 \rangle\cap \Delta_+$.
Let $\fg_2$ be the corresponding rank two Lie algebra. Assume that $\{\gamma_1, \gamma_2\}$ form a
basis of the semigroup of roots $\mathbb{Z}\langle \gamma_1, \gamma_2 \rangle\cap \sigma(\Delta_+)$.
Let $\alpha_1', \alpha_2'$ be a basis of $\Delta_2\cap \Delta_+$; we consider $\alpha_1', \alpha_2'$ as simple roots.
We denote by $\omega_1'$, $\omega_2'$ the corresponding fundamental weights.
Let $\sigma_2$ be an element of the Weyl group $W_2$ of the root system $\Delta_2$
such that $\gamma_i=\sigma_2(\alpha_i)$, $i=1,2$.
Let $\I_2$ be the Iwahori algebra of the Lie algebra $\fg_2$, that can be seen as $\fg_2 [z] \cap \I \subset \fg [z]$.
Let $\mathbb{M}_2$ be the cyclic $\I_2$-module with the generator $w_2$ and the following set of relations:
\begin{equation}\label{rank2wanishing}
e_{\widehat{\sigma_2}(-\alpha+\delta)+r\delta}w_2=0, \alpha \in \Delta_{2+}, r \geq 0;
\end{equation}
\begin{equation}\label{rank2l}
\widehat \sigma_2 (e_{\alpha})^{l_{\alpha,m-1}}w_2=0, \alpha \in \Delta_{2+}.
\end{equation}
Let ${M}_2$ be the cyclic $\I$-module defined by relations \eqref{rank2wanishing}, \eqref{rank2l} and the following relation:
\begin{equation}\label{rank2cartan}
(h \otimes t^k) w_2=0, h \in \fh, k \geq 1.
\end{equation}


\begin{lem}\label{ranktworeduction}
Assume that $-\bar \beta_m \in \mathbb{Q} \Delta_2$.\\
(i)There exists $k \in \mathbb{N}$ such that ${M}_2$ $(\mathbb{M}_2)$ is isomorphic to the following local $($global$)$ Weyl module with characteristic
$$W_{\sigma_2\left((l_{\sigma^{-1}\gamma_1,m-1}+1)\omega'_1+l_{\sigma^{-1}\gamma_2,m-1}\omega'_2\right)}(k-1)$$
$$\left(\mathbb{W}_{\sigma_2\left((l_{\sigma^{-1}\gamma_1,m-1}+1)\omega'_1+l_{\sigma^{-1}\gamma_2,m-1}\omega'_2\right)}(k-1)\right)$$
defined for the reduced decomposition of the element $t_{-\mu}$, $\mu=\omega_1'$ or in type $G_2$ 
\begin{equation}\label{remainingg2}
W_{\sigma_2\left((l_{\sigma^{-1}\gamma_1,m-1}+1)\omega'_1+(l_{\sigma^{-1}\gamma_2,m-1}+1)\omega'_2\right)}(6)
\end{equation}
defined for the reduced decomposition of the element $t_{-\mu}$, $\mu=\omega_1'+\omega_2'$ with the reduced decomposition such that
$\{-\bar \beta_1,\dots, -\bar \beta_6\}=\Delta_{2+}$.

(ii)Let $\beta'_i$ be the set of $\beta$'s for the decomposition with respect to $t_{-\mu}$.
If $-\bar \beta_m \in \Delta_2$, then $-\bar \beta_m=-\bar \beta'_k$.
\end{lem}

\begin{proof}
We work out the case of local Weyl module, the global case can be worked out in the same way.
We need to prove that $M_2$ satisfies
all defining relations of $$W_{\sigma_2\left((l_{\alpha_1}+1)\omega_1+l_{\alpha_2}\omega_2\right)}(k-1).$$
Assume first that our reduced decomposition of $t_{-\mu}$ is a concatenation of reduced decompositions of
several elements of the form $t_{-\omega_j'}$. Then the sequence $-\bar \beta_i$ is the concatenation of such sequences for fundamental weights.
Thus the Lemma follows from \cite{FeMa3}, Lemma 2.17.

Note that any two reduced decompositions of an
element of a Coxeter group can be connected by a sequence of Coxeter relations. Assume that for some reduced decomposition
of $t_{-\mu}$ the claims hold. We prove the claims for a reduced decomposition which differs by one Coxeter relation:
\[s_{i_{r+1}}\dots s_{i_{r+s}}=s_{i_{r+1}}'\dots s_{i_{r+s}}'.\]

Lemma \ref{braidonbetas} tells us that the sequence of $\bar \beta_i$ is changed by the permutation
$(r+1,r+s)(r+2,r+s-1),\dots$. Therefore if $r+s \leq m-1$  or $r+1 \geq m$ then
the claims still hold by the obvious reason.

Assume that $r+s>m-1$, $r+1<m$. If $\mathbb{Q}\Delta_2 \neq \mathbb{Q}\langle \bar \beta_{r+1}, \bar \beta_{r+2}\rangle$ then
the only $m'\in\{ r+1, \dots, r+s\}$ such that $-\bar\beta_{m'}\in \Delta_2$ is $m'=m$.
Therefore the set of modules ${M}_2$ (for all $m \in \{1, \dots, l(t_{-\mu})\}$) does not change and the proof is completed.

Therefore we only need to consider the case $\mathbb{Q}\Delta_2 = \mathbb{Q} \langle \bar \beta_{r+1}, \bar \beta_{r+2}\rangle$, i. e.
to consider the case of rank two Lie algebra.
Assume that we apply the Coxeter relation of the algebra $\fg_2$. Then:
\begin{equation}\label{coxlengthequality}
l(s_{i_1}\dots s_{i_r}s_{i_{r+1}})=l(s_{i_1}\dots s_{i_r}s_{i_{r+2}})=l(s_{i_1}\dots s_{i_r})+1.
\end{equation}
Let $\pi's_{i_1}\dots s_{i_r}=t_{-\nu}\kappa$ for some $\pi' \in \Pi$, $\nu \in P$, $\kappa \in W$.
More precisely,  $\nu \in P_+$ because the reduced decomposition of $t_{-\nu}$ is a truncation of a reduced decomposition of $t_{-\mu}$.
Thus \eqref{coxlengthequality} implies that $\kappa={\rm id}$, i. e. $\pi's_{i_1}\dots s_{i_r}=t_{-\nu}$ for some $\nu \in P_+$, $\pi' \in \Pi$.
(In the language of \cite{OS} this means that the alcove $s_{i_1}\dots s_{i_r}$ has two walls labeled
by roots $\alpha_1^\vee+a_1\delta$, $\alpha_2^\vee+a_2\delta$ for some $a_1, a_2\in -\mathbb{N}$ and this alcove is on the positive side
of these walls. This implies the needed equality.)
Recall that in this case $-\bar\beta_{r+1}, \dots, -\bar\beta_{r+s}$ is the sequence of all positive coroots in the
convex order and the sequence $-\bar\beta_{r+1}, \dots, -\bar\beta_{r+k}$, $k<s$, coincides with the sequence of first $k$ elements $-\bar \beta_i$
for the decomposition of $t_{-\omega_i'}$ (see \cite{FeMa3}, Section 3). This completes the induction step for $\Delta_2$ of type $A_2$.
Indeed, in type $A_2$ there is no subsystem of rank $2$ and therefore there are no other relations in the affine Weyl group.

In type $C_2$ the only remaining case is the case of two long coroots $-\bar\beta_{m-1}, -\bar\beta_m$.
We have
$$l(s_{i_1}\dots s_{i_{m-2}}s_{i_{m-1}})=l(s_{i_1}\dots s_{i_{m-2}}s_{i_{m}})=l(s_{i_1}\dots s_{i_{m-2}})+1.$$
Analogously to the previous case we obtain that
 $s_{i_1}\dots s_{i_{m-2}}=\pi t_{-\nu}s_{\alpha_1'}s_{2\alpha_1'+\alpha_2'}$ and the sequence $\bar \beta_i$ is the set
of $\bar \beta$'s for $t_{-\nu}$ for $i \leq m-4$ and
$$-\bar \beta_{m-3}={\alpha_1'}^{\vee},-\bar \beta_{m-2} = (2\alpha_1'+\alpha_2')^{\vee}.$$
(in the language of \cite{OS} this means that the alcove $s_{i_1}\dots s_{i_{m-2}}$ has two walls labeled
by roots $\alpha_1^\vee+a_1\delta$, $(\alpha_1+\alpha_2)^\vee+a_2\delta$ for some $a_1, a_2\in -\mathbb{N}$ and this alcove is on the positive side of these walls).
If $-\bar\beta_{m-1}=(\alpha_1'+\alpha_2')^\vee$, then $M_2$ is isomorphic to the generalized Weyl module (without characteristic)
$W_{\sigma_2\left(l_{\sigma^{-1}\gamma_1,m-1}\omega'_1+l_{\sigma^{-1}\gamma_2,m-1}\omega'_2\right)}$.
In the remaining case $-\bar\beta_{m-1}={\alpha_1'}^\vee,-\bar\beta_{m}=(\alpha_1'+\alpha_2')^\vee$.
Put $m_1:=l_{m-1,\alpha_1'}$, $m_2:=l_{m-1,\alpha_2'}$.
Then $l_{m-1,\alpha_1'}=(m_1+2)-2$, $l_{m-1,2\alpha_1'+\alpha_2'}=(m_1+2+m_2)-1$, $l_{m-1,\alpha_1'+\alpha_2'}=m_1+2+m_2$,
$l_{m-1,\alpha_2'}=m_2$. However:
\begin{multline*}
\big((m_1+2)-2,(m_1+2+m_2)-1,m_1+2+2m_2,m_2\big)\\=\big(m_1,m_1+m_2+1,m_1+2m_2+2,(m_2+1)-1\big).
\end{multline*}
Therefore $M_2\simeq W_{\sigma_2\left(l_{\sigma^{-1}\gamma_1,m-1}\omega'_1+(l_{\sigma^{-1}\gamma_2,m-1}+1)\omega'_2\right)}(1)$
defined by the reduced decomposition of the element $t_{-\omega_2'}$.
This completes the proof for $C_2$.

Analogously we prove that any generalized Weyl module with characteristic (with respect to
a decomposition by arbitrary element) in type $G_2$ is isomorphic to the generalized Weyl module
with characteristic with respect to the decomposition by $t_{-\omega_i'}$ or to the module \eqref{remainingg2}.
Note that in this case we have three types of Coxeter relations. One of them is the Coxeter relation of type $G_2$,
the second is of type $A_2$ (which acts on the sequence of $-\bar\beta$'s in the following way: it takes a subsequence of
three long coroots $\alpha_1^\vee, (\alpha_2+2\alpha_1)^\vee, (\alpha_2+\alpha_1)^\vee$ and interchanges the first and the third coroots
 of this subsequence)
and of type $A_1\times A_1$ (which interchanges two orthogonal coroots in the sequence of $-\bar\beta$'s). We consider only the relations
of type $A_2$, the  remaining case can be considered in the same way. 
Assume that $-\bar\beta_{r+1}=\alpha_1^\vee$, $-\bar\beta_{r+2}= (\alpha_2+2\alpha_1)^\vee$, $-\bar\beta_{r+3}=(\alpha_2+\alpha_1)^\vee$
and we can apply this type Coxeter relation. Then we have:

$$l(s_{i_1}\dots s_{i_{r}}s_{i_{r+1}})=l(s_{i_1}\dots s_{i_{r}}s_{i_{r+3}})=l(s_{i_1}\dots s_{i_{r}})+1.$$

Therefore for some $\nu \in P_+$
 $$s_{i_1}\dots s_{i_{r}}=t_{-\nu}s_1s_2s_1s_2 \text{ or } 
s_{i_1}\dots s_{i_{r}}=t_{-\nu}s_2s_1s_2s_1s_2.$$
(In the language of \cite{OS} this means that the alcove $s_{i_1}\dots s_{i_{m-2}}$ has two walls labeled
by roots $\alpha_1^\vee+a_1\delta$, $(\alpha_1+\alpha_2)^\vee+a_2\delta$ for some $a_1, a_2\in -\mathbb{N}$ and this alcove is on the positive side of these walls).
In the first case $-\bar\beta_{r-3}=\alpha_1^\vee$, $-\bar\beta_{r-2}=(\alpha_2+3\alpha_1)^\vee$, $-\bar\beta_{r-1}=(\alpha_2+2\alpha_1)^\vee$,
$-\bar\beta_{r}=(2\alpha_2+3\alpha_1)^\vee$. 
Then if $-\bar\beta_{r+1}=(\alpha_2+\alpha_1)^\vee$, $-\bar\beta_{r+2}= (\alpha_2+2\alpha_1)^\vee$, $-\bar\beta_{r+3}=\alpha_1^\vee$
or $-\bar\beta_{r+1}=\alpha_1^\vee$, $-\bar\beta_{r+2}= (\alpha_2+2\alpha_1)^\vee$, $-\bar\beta_{r+3}=(\alpha_2+\alpha_1)^\vee$ then all the modules 
$W_{\sigma(\lambda)}(r+j)$, $j=1,2,3$ are isomorphic to the generalized Weyl modules with characteristic defined by $t_{-\omega_i'}$.

In the second case $-\bar\beta_{r-4}=\alpha_2^\vee$, $-\bar\beta_{r-3}=(\alpha_2+\alpha_1)^\vee$, 
$-\bar\beta_{r-2}=(2\alpha_2+3\alpha_1)^\vee$, 
$-\bar\beta_{r-1}=(\alpha_2+2\alpha_1)^\vee$,
$-\bar\beta_{r}=(\alpha_2+3\alpha_1)^\vee$. Then if 
$-\bar\beta_{r+1}=\alpha_1^\vee$, $-\bar\beta_{r+2}=(\alpha_2+2\alpha_1)^\vee$, $-\bar\beta_{r+3}=(\alpha_2+\alpha_1)^\vee$
then we have that $W_{\sigma(\lambda)}(r+1)$ is isomorphic to the module \eqref{remainingg2}. In both cases all remaining modules
are isomorphic to the generalized Weyl modules with characteristic defined by $t_{-\omega_i'}$.
\end{proof}

\begin{prop}\label{pre-decompositionlocal}
The defining equations of $W_{\sigma(\la_-)}(m-1)$ $(\mathbb{W}_{\sigma(\la_-)}(m-1))$
 impose the following equations on $w _{\sigma ( \la_-), m-1}$:
\begin{gather*}\label{charrel-dp}
e_{{\widehat{\sigma s_{\bar\beta_m}}}(\al)}^{l_{\al,m}+1} w _{\sigma ( \la_-), m-1} = 0, \al \in\Delta_+\\
\left(e_{{\widehat{\sigma s_{\bar\beta_m}}}(\al)}^{l_{\al,m-1}+1} w _{\sigma ( \la_-), m-1} = 0, \al \in \Delta_+\right).
\end{gather*}
In particular, $W_{\sigma(\la_-)}' (m)$ ($\mathbb{W}_{\sigma(\la_-)}' (m)$) is a quotient of the module
$W_{\sigma s_{\bar\beta_m} (\la_-)} (m)$ $(\mathbb{W}_{\sigma s_{\bar\beta_m} (\la_-)} (m-1))$ as $\fn^{af}$-modules.
\end{prop}

\begin{proof}
The claim follows from Lemma \ref{ranktworeduction} for $\Delta_2$ spanned by $\al$, $\bar\beta_m^\vee$ and computations
from \cite{FeMa3}, Section $3$ (we note that $l_{\al,m}=l_{\al,m-1}$ if $\alpha^\vee \neq -\bar \beta_{m}$ and
$l_{\al,m}=l_{\al,m-1}-1$ if $\alpha^\vee = -\bar \beta_{m}$).
\end{proof}

\begin{prop}\label{decompositionlocal}
Assume there exists an arrow $\sigma \to \sigma s_{\bar\beta_m}$ in QBG. Then the kernel of the surjection
$W_{\sigma(\la_-)}(m-1) \to W_{\sigma(\la_-)}(m)$ is a quotient of $W_{\sigma s_{\bar\beta_m}(\la_-)}(m)$.
If the arrow does not exist in QBG, then the surjection $W_{\sigma(\la_-)}(m-1) \to W_{\sigma(\la_-)}(m)$ has no kernel.
\end{prop}

\begin{proof}
By the above calculations, we can transplant the relations by examining each rank two root subsystem containing $\bar\beta_m$. 
The details are completely analogous to the proof of \cite[Theorem 2.18]{FeMa3}.
\end{proof}

The following results are what we call the decomposition procedure, that is originally proved in \cite{FeMa3} when 
$\mu = -\omega_i$, $i \in I$: their proofs will be given after Corollary \ref{combmodel}.

\begin{thm}\label{kernel}
If the kernel of the surjection
$W_{\sigma(\la_-)}(m-1) \to W_{\sigma(\la_-)}(m)$ is non-trivial, then it is isomorphic to $W_{\sigma s_{\bar\beta_m}(\la_-)}(m)$.
\end{thm}

\begin{thm}\label{local}
Let $\mu$ be an anti-dominant weight with a reduced decomposition \eqref{pi} and $\la_-, \la_--\mu\in P_-$.
For each $0\le m\le l=\ell(t_\mu)$, the generalized Weyl module with characteristics $W_{\sigma(\la_-)}(m)$, constructed
via a reduced decomposition of $t_\mu$, can be filtered in such a way that:
\begin{itemize}
\item each subquotient is a generalized Weyl module of the form $W_{\tau(\la_--\mu)}$ for some Weyl group
element $\tau$;
\item
the number of subquotients is equal to the number of directed paths in the quantum Bruhat graph starting at $\sigma$ and
with labels of the form $\bar\beta_{j_1},\dots,\bar\beta_{j_k}$, $m\le j_1<\dots<j_k\le l$.
\end{itemize}
\end{thm}

We consider a reduced decomposition of $t_{\la_-}$ obtained by concatenating the reduced decomposition of $t_\mu$ (used to construct
$W_{\sigma(\la_-)}(m)$) and a reduced decomposition of $t_{\la_--\mu}$.
For each $0 \le m < \ell(t_{\la_-})$, we define ${\mathcal{QB}}_{\sigma,\la_-}(m)$ to be the subset of ${\mathcal{QB}}(id; t_{\sigma(\la_-)}\sigma)$ (see section \ref{qBg}) so that the sequence of $\beta$'s given by
$$\beta_{m+1}(\la_-),\dots,\beta_l(\la_-), \ \text{where} \ l=\ell(t_{\la_-}).$$

\begin{cor}\label{combmodel}
Under the above settings, we have:
\[
\ch \, W_{\sigma(\la_-)}(m)=\sum_{p\in {\mathcal{QB}}_{\sigma,\la_-}(m)} x^{{\rm wt}({\rm end} (p))} q^{\deg(qwt^*(p))}.
\]
\end{cor}

\begin{proof}[Proofs of Theorem \ref{kernel}, Theorem \ref{local}, and Corollary \ref{combmodel}]
By a repeated application of Proposition \ref{decompositionlocal}, we deduce a numerical inequality version of Theorem \ref{local}, that asserts 
$\ch \, W_{\sigma(\la_-)}(m)$ is smaller than or equal to the sum of $\ch \,  W_{\tau(\la_--\mu)}$ described in Theorem \ref{local}. Moreover, an equality here implies Theorem \ref{local} itself.

For $\mu=\la_-$, every subquotient in Theorem \ref{local} is isomorphic to $W _0$ (with weight twists), and hence is one-dimensional 
(contributes by one to the numerical inequality in the previous paragraph). In addition, these subquotients are parametrized by 
the elements $p\in {\mathcal{QB}}_{\sigma,\la_-}(m)$.

Hence, if we know that the character of the one-dimensional subquotient labeled by
$p$ is given by $x^{{\rm wt}({\rm end} (p))} q^{\deg(qwt^*(p))}$ (that calculates the effect of the change of weights in 
Proposition \ref{pre-decompositionlocal}, and also the equality in Corollary \ref{combmodel} when $m=0$), then we deduce the 
equality in our numerical version of Theorem \ref{local} in this particular case. They are contained in \cite[Corollary 2.9]{FeMa3} and \cite[Theorem 2.21]{FeMa3}, respectively.

We apply the numerical inequality version of Theorem \ref{local} repeatedly to conclude the equality in the end. 
Thus the $m = 0$ case verifies the equality in the numerical inequality version of Theorem \ref{local} for general $m$. 
Hence, Theorem \ref{kernel} and Theorem \ref{local} hold. Now Corollary \ref{combmodel} for general $m$ follows as every $p\in {\mathcal{QB}}_{\sigma,\la_-}(m)$ contributes to the character of $\ch \, W_{\sigma(\la_-)}(m)$. 
\end{proof}

Below we present the decomposition procedure for the global module $\W_{\sigma(\la_-)}(m)$. We assume that the weight of the cyclic vector of
$\W_{\sigma(\la_-)}(m)$ is equal to $\sigma(\la_-)$ unless stated otherwise.

The character of a finitely generated graded $\I$-module $M$ is well-defined as:
$$\ch \, M = \sum_{\mu\in P} x^\mu c_\mu(q) \hskip 10mm c_\mu(q) \in \mathbb Z [\![q]\!].$$
In particular, $\ch \, W_{\sigma(\la_-)}(m)$ and $\ch \, \W_{\sigma(\la_-)}(m)$ make sense for every $\sigma \in W, \la_- \in P_-,$ and $m \in \mathbb Z_{\ge 0}$ as the both modules are cyclic.

The following is a slight generalization of \cite{FMO} from the case $\mu = - \omega_i$.

\begin{thm}\label{decompositionglobal}
Let $\la_-,\mu, \la_- - \mu \in P_-$ and let $\beta_j=\beta_j(\mu)$. Then the following holds:
\begin{enumerate}
\item If there is no edge $\sigma \to \sigma s_{\bar\beta_m}$ in QBG, then the surjective map
$\W_{\sigma(\la_-)}(m-1)\to {\mathbb W}_{\sigma(\la_-)}(m)$ is an isomorphism.
If the edge does exist, then the kernel of this map is isomorphic to ${\mathbb W}_{\sigma s_{\bar\beta_m}(\la_-)}(m-1)$;
\item \label{characters} One has a character equality
\[
\ch \, \W_{\sigma(\la_-)}(m)=\frac{\ch \, W_{\sigma(\la_-)}(m)}{(q)_{\la_- + \omega(m)}},
\]
where we set
$$\omega ( m ) := \sum_{1 \le i \le m, -\bar\beta_i=\al_{j}^\vee \text{ is simple}} \omega_{j}.$$
\end{enumerate}
\end{thm}
\begin{proof}
The proof of the first claim of $i)$ is analogous to that of \cite[Theorem 2.18 i)]{FeMa3}. We deduce that the kernel of the map
$\W_{\sigma(\la_-)}(m-1)\to {\mathbb W}_{\sigma(\la_-)}(m)$ is a quotient
 of the module
${\mathbb W}_{\sigma s_{\bar\beta_m}(\la_-)}(m-1)$ by the same way as in Proposition \ref{decompositionlocal}
(see Proposition \ref{pre-decompositionlocal}).

To prove part \eqref{characters} we use the same inductive argument as in the proof \cite[Theorem 3.16]{FMO}
(that ultimately relies on the counting paths in the quantum Bruhat graph in \cite[Theorem 2.21]{FeMa3} through \cite[Lemma 3.12]{FMO}).
The only modification needed is \cite[Lemma 3.13]{FMO}. Namely, the crucial point in this Lemma is to figure
out if a coroot $-\bar\beta_{\bullet}$ is simple. In particular, the proof of \cite[Theorem 3.16]{FMO} implies that
if
\[
\ch \, \W_{\sigma(\la_-)}(j)=\frac{\ch \, W_{\sigma(\la_-)}(j)}{(q)_{\mu}}
\]
for some anti-dominant weight $\mu$, then
\[
\ch \, \W_{\sigma(\la_-)}(j+1)=
\begin{cases}
\frac{\ch \, W_{\sigma(\la_-)}(j+1)}{(q)_\mu}, & \text{ if } -\bar\beta_{j+1}\text{ is not simple},\\
\frac{\ch \, W_{\sigma(\la_-)}(j+1)}{(q)_{\mu+\om_i}}, &  \text{ if } -\bar\beta_{j+1}=\al^\vee_i.
\end{cases}
\]
If $\mu$ is equal to a negated fundamental weight, then $-\bar\beta_{j+1}$ is simple if and only if
$j=0$. The main difference here and \cite{FMO} is that there might be several simple roots among $\{-\bar\beta_{j} \}^{m}_{j=1}$ for general $\mu$.
\end{proof}

\subsection{Nonsymmetric Macdonald polynomials at infinity}\label{NM}
We fix $\la \in P_-$, and we assume that $\sigma \in W$ is the maximal length element in the class $\sigma\cdot {\rm stab}_W
(\la_-)$. We set $\la'_-=w_0\sigma(\la_-)$. Then, $v (\la_-')=\sigma^{-1}w_0$ is the shortest element such that
$v ( \la_-' ) \la_-'=\la_-$ (see the end of section \ref{Aa}). Moveover, the factorization $t _{\la_-} = v ( \la_-' ) u(\la_-')$
refines to a reduced expression.

If we fix reduced expressions
\[
v ( \la_-' )=s_{i_1}\dots s_{i_r},\ u(\la_-')=\pi s_{i_{r+1}}\dots s_{i_M},
\]
then we obtain a reduced expression
\begin{equation}\label{reddec}
t_{\la_-}=\pi s_{\pi^{-1}i_1}\dots s_{\pi^{-1} i_r} s_{i_{r+1}}\dots s_{i_M}.
\end{equation}

We apply the procedure of section \ref{gWc} to the reduced decomposition \eqref{reddec} to obtain a sequence $\beta_j=\beta_j(t_{\la_-})$ for $1 \le j \le \ell (t_{\la_-})$ that we fix throughout this section.

\begin{thm}\label{UWchar}
Under the above settings, we have
$$U_{\sigma(\la_-)} \simeq W_{\sigma(\la_-)}(\ell(w_0)-\ell(\sigma)).$$
In addition, we have
$\ch \, W_{\sigma(\la_-)}(\ell(w_0)-\ell(\sigma))=w_0 E_{w_0\sigma\la_-}(x,q^{-1},\infty)$.
\end{thm}

\begin{rem}
The modules $W_{\sigma(\la_-)}$ depend only on $\sigma(\la_-)$, but not on $\sigma$ and $\la_-$
separately. Therefore, our choice of $(\sigma, \lambda_-)$'s cover the whole of $P$ is a bijective fashion.
\end{rem}

\begin{proof}[Proof of Theorem \ref{UWchar}]
We have to prove that the defining relations \eqref{charrel} of $W_{\sigma(\la_-)}(\ell(w_0)-\ell(\sigma))$ coincide
with the defining relations of $U_{\sigma(\la_-)}$. We set $r:=\ell(v(\la_-'))=\ell(\sigma^{-1}w_0)=\ell(w_0)-\ell(\sigma)$,
that is the cardinality of the set $\Delta_+\cap\sigma^{-1} \Delta_+$.

It suffices to show that $\{-\bar\beta_1^\vee,\dots,-\bar\beta_r^\vee\}=\Delta_+\cap\sigma^{-1} \Delta_+$. By definition, for $k=1,\dots,r$ we
have
\begin{eqnarray}
\beta_k & = & s_{i_M}\dots s_{i_{r+1}} s_{\pi^{-1}i_r}\dots s_{\pi^{-1}i_{k+1}} \al_{\pi^{-1}i_k}^\vee\\
&= & t_{-\la_-} \pi s_{\pi^{-1}i_1}\dots s_{\pi^{-1}i_{k-1}} s_{\pi^{-1}i_k}  \al_{\pi^{-1}i_k}^\vee \\
&= & t_{-\la_-} s_{i_1}\dots s_{i_{k-1}} (-\al_{i_k}^\vee).
\end{eqnarray}
The action of $t_{-\la_-}$ on $\Delta^{a}$ preserve the finite (bar) part of an affine coroot. Therefore, the negated finite parts of $\beta_1^\vee,\dots,\beta_r^\vee$ are exactly the positive roots which are mapped to negative roots by $w_0\sigma$. Therefore, the comparison of the defining equations yield
$$U_{\sigma(\la_-)} \simeq W_{\sigma(\la_-)}(r),$$
that is the first part of the assertion.

The Orr-Shimozono formula (\cite{OS}, Proposition 5.4) for the $t=\infty$ specialization of the
non-symmetric Macdonald polynomials asserts:
\begin{equation}\label{OS}
E_{\la_-'}(x,q^{-1},\infty)=\sum_{p\in \overleftarrow{\mathcal{QB}}(id; u(\la_-'))} x^{wt(p)} q^{\deg(qwt^*(p))} \hskip 10mm \la_-'\in P.
\end{equation}
Here we use the reduced decomposition $u({\la_-'})=\pi s_{i_{r+1}}\dots s_{i_M}$ and the corresponding coroots
$\beta_j$ in the definition of $\overleftarrow{\mathcal{QB}}(id; u(\la_-'))$.
In other words, $E_{\la_-'}(x,q^{-1},\infty)$ is equal to the sum over all paths in the reversed quantum Bruhat
graph with $z_0=u({\la_-'})$.
We note that ${\rm dir} (u({\la_-'}))= {\rm dir} (t_{\la_-'}v(\la_-')^{-1})=w_0\sigma$. We can pass an alcove path on the
reversed quantum Burhat graph to an alcove path on QBG by the left multiplication of $w_0$. Hence, the sum in \eqref{OS} multiplied by $w_0$ from the left
ranges over all paths in QBG starting at $\sigma$.

Corollary \ref{combmodel} gives the combinatorial formula of $\mathrm{ch} \, W_{\sigma(\la_-)}( r )$, that is identical to $w_0E_{w_0\sigma\la_-}(x,q^{-1},\infty)$ by \eqref{OS} through the above identification. Hence we
obtain
$$\ch \, W_{\sigma(\la_-)}(r)=w_0E_{w_0\sigma\la_-}(x,q^{-1},\infty)$$
as required.
\end{proof}

\begin{cor}\label{UE}
The character of $U_{\sigma(\la_-)}$ is equal to $w_0E_{w_0\sigma(\la_-)}(x,q^{-1},\infty)$.
Equivalently, we have $E_{\sigma(\la_-)}(x,q^{-1},\infty)=w_0\ch \, U_{w_0\sigma(\la_-)}$.
\end{cor}

\begin{cor}\label{lambdasigma}
The algebra $A_{({\lambda_-}_\sigma)}$ acts freely on ${\mathbb U}_{\sigma(\la_-)}$ and
\[
\ch \, {\mathbb U}_{\sigma(\la_-)}= \ch \, U_{\sigma(\la_-)}/(q)_{({\lambda_-})_\sigma}.
\]
\end{cor}
\begin{proof}
Theorem \ref{UWchar} and its proof imply
\[
\ch \, {\mathbb U}_{\sigma(\la_-)}=\ch \, \W_{\sigma(\la_-)}(\ell(w_0\sigma))=
\frac{\ch \, W_{\sigma(\la_-)}(\ell(w_0\sigma))}{(q)_{\nu}},
\]
where $\nu$ is obtained from $\la_-$ by adding all fundamental weights $\omega_j$ such that
the corresponding simple roots $\al_j$ show up as $-\bar\beta_i^\vee$ for $i=1,\dots,\ell(w_0)-\ell(\sigma)$
(see Theorem \ref{decompositionglobal}, \eqref{characters}). Such $\al_j$ are exactly the simple
roots mapped to $\Delta_+$ by $\sigma$. Hence, we conclude that $\nu=(\la_-)_\sigma$. This shows the character equality. The rest of the assertion is Corollary \ref{U-free}.
\end{proof}

\section{Global U-modules and sheaves on semi-infinite Schubert varieties}\label{SoSIS}
Let $\Q$ be the semi-infinite flag variety (see \cite{FiMi},\cite{BF1}). For an element $w\in W$
we denote by $\Q(w)\subset\Q$ the corresponding semi-infinite Schubert variety (see \cite{Kat}).
The varieties $\Q(w)$ are defined as follows. Let $X(w)\in G/B$ be the (finite-dimensional)
Scubert variety corresponding to the element $w$. Let ${\rm ev}_0:\Q_0\to G/B$ be the evaluation map
from the subvariety $\Q_0\subset\Q$ of no-defect quasi-maps to the flag variety of $G$.
By definition, $\Q(w)=\overline{{\rm ev}^{-1} (X(w))}$. In particular, we have an embedding $X ( w ) \subset
\Q(w)$ consisting of constant loops. Let us denote the unique $H$-fixed point of the dense open $B$-orbit of
$X(w)$ by $x_{w}$. We regard $x_{w}$ as a point in $X (w) \subset \Q(w)$.

\begin{rem}
The contents of this section can be also formulated by employing the formal model $\mathbf Q$ of the semi-infinite flag variety (instead of the ind-model) defined in \cite[\S 4.1]{FiMi} by assuming the results from \cite{BF1} and \cite[\S 4]{KNS}.
\end{rem}

We note that each semi-infinite Schubert variety inherits an ind-structure from $\Q$, i.e.
$\Q(w)=\cup_{\beta\in Q^\vee_+} \Q(w,\beta)$. Using the embedding
\[
\Q(w,\beta)\subset \prod_{i=1}^n \bP(V(\omega_i)\otimes \bC[z]_{\le \langle\beta,\omega_i\rangle})
\]
one gets for each $\lambda\in P$ the line bundle $\cO_w(\la)$ on  $\Q(w)$
(this is the projective limit of the line bundles $\cO_{w,\beta}(\la)$ on  $\Q(w,\beta)$).
We define the $i$-th cohomology of $\cO_{w}(\la)$ by
\[
H^i(\Q(w),\cO_w(\la)) :=\left( \varprojlim_{\beta}H^i(\Q(w,\beta),\cO_{w,\beta}(\la))
\right)_{\Gm\text{-finite}}.
\]

It is proved in \cite[Theorem 4.12]{Kat} that for $\la\in P_+$ one has
\begin{equation}
H^0(\Q(w),\cO_w(\la))^*\simeq {\mathbb W}_{w \la},\label{GammaO}
\end{equation}
where $*$ denotes the restricted dual and all the higher cohomologies vanish. Let us denote by $u_{w\la}$ the
$\I$-cyclic generator of $H^0(\Q(w),\cO_w(\la))^*$ that is fixed by the action of the loop rotation (such a vector is unique up to constant).
 In \cite[\S 6]{Kat}, the author constructs sheaves $\E_{w}(\la)$ on $\Q(w)$ such that
\[
\ch \, H^i(\Q(w),\E_{w}(\la))^*=\delta_{i,0} E_{-w(\la)}(x^{-1},q^{-1},\infty) \in \bC [P][\![q]\!],
\]
holds for each $\la \in P$, where $x^{-1}$ means the replacement of $e^\mu$ by $e^{-\mu}$ for each $\mu \in P$. Moreover, $H^0(\Q(w),\E_{w}(\la))^*$ is a cyclic $\I$-module (\cite[Lemma 6.7]{Kat}).

\begin{cor}
For a dominant weight $\la$ and $w\in W$ one has
\begin{equation}\label{GammaE}
\ch \, \Gamma(\Q(w),\E_{w}(\la))^*=\ch \, {\mathbb U}_{w\la}.
\end{equation}
\end{cor}
\begin{proof}
We set that $\sigma=ww_0$, and $\la_-=w_0\la$. Remark \ref{><} implies $(q)_{\la_w}=(q)_{(\la_-)_\sigma}$. By \cite[Corollary 6.10]{Kat}, we have an equality:
\begin{equation}\label{Gamma}
\ch \, \Gamma(\Q(w),\E_{w}(\la))^*=(q)_{\la_w}^{-1} \cdot E_{-w\la}(x^{-1},q^{-1},\infty).
\end{equation}
We see that
\[
E_{-w\la}(x^{-1},q^{-1},\infty)=w_0E_{w_0w\la}(x,q^{-1},\infty) = w_0E_{w_0\sigma\la_-}(x,q^{-1},\infty),
\]
where the first equality is \cite[Lemma 5.2]{OS}, and the second equality is by convention.

By Corollary \ref{UE}, we have
\[
\ch \, U_{w\la}=\ch \, U_{\sigma\la_-}=w_0E_{w_0\sigma\la_-}(x,q^{-1},\infty).
\]
Corollary \ref{lambdasigma} tells us that
$\ch \, {\mathbb U}_{\sigma\la_-}=(q)_{(\la_-)_\sigma}^{-1} \cdot \ch \, U_{\sigma\la_-}$. Using \eqref{Gamma}
and Remark \ref{><} we conclude that \eqref{GammaE} holds true.
\end{proof}

We briefly recall the construction of the sheaves $\E_{w}(\la)$ from \cite{Kat}.
Let $w=s_{i_1}\dots s_{i_l}$ be a reduced decomposition of $w$. Let ${\bf I}_k\supset {\bf I}$ be
the parabolic subgroup corresponding to $\alpha_{k}$ that contains the Iwahori group ${\bf I} \subset G ( \bC ( z ) )$. We define
\[
\Q({\bf i})={\bf I}_{i_1}\times_{\bf I}{\bf I}_{i_2}\times_{\bf I}\dots \times_{\bf I}{\bf I}_{i_l}\times_{\bf I} \Q(e),
\]
where the last factor is the smallest semi-infinite Schubert variety corresponding to the identity element $e\in W$. We set
\[
\gamma_{1} := \al_{i_{1}}, \gamma_{2} := s_{i_{1}} \al_{i_{2}}, \ldots, \gamma_{l} := s_{i_{1}} s_{i_{2}} \cdots
s_{i_{l-1}} \al_{i_{l}}.
\]
The roots $\gamma_i$ are distinct to each other and each of them belongs to $\Delta_{+}$ since our choice of $\bi$ is reduced.
Note that if we have a subexpression ${\bi}'$ of $\bi$, then we have natural embedding $\Q ({\bi}')
\hookrightarrow \Q ({\bf i})$ of the analogously defined variety by understanding that the elements from the missing factors $i_{j}$ to be belong to ${\bf I} \subset {\bf I}_{i_{j}}$. This particularly induces an inclusion $\Q ( e ) = \Q (\emptyset)
\hookrightarrow \Q ({\bf i})$. Hence, we can regard $x_{e}$ also as a point of $\Q ({\bf i})$.

One has the multiplication map \[q_{{\bf i}}:\Q({\bf i})\to \Q(w).\]
For each $1\le k\le l = \ell ( w )$, we consider the divisor $H_k\subset \Q({\bf i})$ defined by
\[
H_k=\{(g_1,\dots,g_l,x)\in \Q({\bf i}),\ g_k\in {\bf I}\}.
\]
Then the sheaf $\E_{w}(\la)$ on $\Q(w)$ is obtained by twisting the line bundle corresponding to $\la$ by the
divisors $H_k$ and pushing it down. Namely, we have
\[
\E_w(\la)=(q_{{\bf i}})_*\cO_{\Q({\bf i})}(\la-\sum_{j=1}^l H_k)=
(q_{{\bf i}})_*\cO_{\Q({\bf i})}(-\sum_{j=1}^l H_k)\otimes \cO_w(\la).
\]
We note that the sheaves $\E_w(\la)$ do not depend on the reduced decomposition of $w$ (\cite[Lemma 6.6]{Kat}).

The maps $q_{{\bf i}}$ satisfy the following important properties (\cite[Lemma 6.1 and Corollary 6.5]{Kat}):
\begin{align}\label{qpushO}
{\mathbb R}^k(q_{\bi})_*\cO_{\Q(\bi)} & =\delta_{k,0}\cO_{\Q(w)},\\\label{qpushE}
{\mathbb R}^k(q_{\bi})_*\cO_{\Q(\bi)}(-\sum_{k=1}^l H_k) &=0,\ k>0.
\end{align}

We also note that the embedding
\[
(q_\bi)_*\cO_{\Q(\bi)}(\la-\sum_{k=1}^l H_k)\subset (q_\bi)_*\cO_{\Q(\bi)}(\la)
\]
gives the embedding $\E_{w}(\la)\hookrightarrow \cO_{\Q_w}(\la)$. Hence \eqref{GammaO} yields an $\I$-module surjection
\begin{equation}\label{WtoG}
{\mathbb W}_{w \la}\to \Gamma(\Q(w),\E_w(\la))^*.
\end{equation}
We conclude that the module $\Gamma(\Q(w),\E_{w}(\la))^*$ is a cyclic $\I$-module that is a quotient
of the generalized global Weyl module ${\mathbb W}_{w\la}$.

\begin{lem}
There exists a surjection of $\I$-modules
\[
{\mathbb U}_{w\la}\to H^{0}(\Q(w),\E_w(\la))^*.
\]
\end{lem}
\begin{proof}
We write $w\la$ as $\sigma \la_-$ for $\sigma=ww_0$, and $\la_-=w_0\la$.
Using the surjection \eqref{WtoG}, we only need to check that the relations
\[
e_{\sigma(\al)}^{-\bra \la_-,\al^\vee\ket}v=0, \ \al\in\Delta_+\cap \sigma^{-1}\Delta_+
\]
hold in $\Gamma(\Q(w),\E_w(\la))^*$, where $v$ is the image of $u_{w\la}$ (these are exactly the relations one
has to add to the defining relations of ${\mathbb W}_{w\la}$ in order to get the module ${\mathbb U}_{w\la}$).
Note that the ($H\times\Gm$-eigen) dual vector of $v$ (or $u_{w\la}$) corresponds to a constant function
$1_{-w\la_{-}}$ on the dense $\bf I$-orbit of $\Q(w)$.

We have an inclusion
$$N_{-\gamma_{1}} \times N_{-\gamma_{2}} \times \cdots \times N_{-\gamma_{l}} \times {\bf I} \subset
{\bf I}_{i_{1}} \times_{\bf I} {\bf I}_{i_{2}} \times_{\bf I} \cdots \times_{\bf I} {\bf I}_{i_{l}},$$
where $N_{\pm \gamma_{j}}$ is the one-dimensional unipotent subgroup of $G(z)$ so that
$\mathrm{Lie} \, N_{\pm \gamma_{j}} \subset \gh$ has $\h$-weight $\pm \gamma_{j}$, respectively.

We consider a curve $\bP^1_{j} \subset \Q(\bi)$ defined as the closure of the affine line $N_{- \gamma_j}x _{e} \subset \Q(\bi)$. We refer this curve as $\bP_{j}^1$ (it is isomorphic to $\bP^{1}$). Since $H$ and $N_{\gamma_{j}}$ fixes $x_{e}$, it follows that $\bP_{j}^1$ is equivariant with respect to the $N_{\gamma_{j}}$-action. Thus, $\bP_{j}^1$ decomposes into the disjoint union of a point $\{ x_{e} \}$ and a $N_{\gamma_{j}}$-orbit isomorphic to $\bA^{1}$.

Our curve $\bP_{j}^1 \subset \Q(\bi)$ is naturally contained in $\Q(\{i_{j},\ldots,i_{l}\})$ so that $\bP_{j}^1 \cap \Q(\{i_{j+1},\ldots,i_{l}\}) = \{ x_{e} \}$. Since $\cO_{\Q(\bi)} (\la)$ is determined by the $H$-character at $x_{e}$ and is equivariant with respect to the group action, it follows that the restriction of $\cO(\la)$ to $\bP^{1}_{j}$ is
$\cO(m)$, where $m=-\bra \la_-,\al^\vee\ket$. The restriction of $H_{k}$ to $\bP_{j}^{1}$ is non-zero if and only if $j=k$, and it defines $\cO (1)$ when $j = k$.

Therefore, we restrict the sheaves $\cO_{\Q(\bi)}(\la)$ and $\cO_{\Q(\bi)}(\la-\sum_{k=1}^l H_k)$
to $\bP^1_{j}$ to obtain the following maps:
\begin{gather*}
H^0(\Q(\bi),\cO_{\Q(\bi)}(\la))\to H^0(\bP^1_{j},\cO(m)),\\
H^0(\Q(\bi),\cO_{\Q(\bi)}(\la-\sum_{k=1}^l H_k))\to H^0(\bP^1_{j},\cO(m-1)).
\end{gather*}
These maps are equivariant with respect to the $N_{\gamma_{j}}$-action. The former map is non-zero since
$1_{-w\la_{-}}$ induces a non-vanishing section of both of them. The section $1_{-w\la_{-}}$ also induces an $\I$-cocyclic vector of $H^0(\Q(\bi),\cO(\la))$ and a $N_{\gamma_{j}}$-cocyclic vector in $H^0(\bP^1,\cO(m))$. By using the embedding
$\cO(\la-\sum_{k=1}^l H_k)\hookrightarrow \cO(\la)$ and dualizing all the pieces, we obtain the commutative
diagram
\[
\xymatrix{
H^0(\bP^1,\cO(m))^* \ar[d]_{\kappa} \ar@{->>}[r] & H^0(\bP^1,\cO(m-1))^*\ar[d]^{\kappa'}\\
{\mathbb W}_{w(\la)}\ar@{->>}[r] & H^0(\Q(w),\E_w(\la))^*
}
\]
from \eqref{GammaO} and \eqref{qpushE}. Here all the spaces have common cyclic vector (with respect to the $N_{\gamma_{j}}$-action in the top line, and with respect to the $\I$-action in the bottom line) induced by $u_{w\la}$.

Note that $U(\fb) u_{w\la} \subset {\mathbb W}_{w(\la)} \subset {\mathbb W}_{\la}$ spans a Demazure submodule of
$\g$. In particular, the span of $\{e_{\gamma_{j}}^{n} u_{w\la}\}_{n \ge 0}$ constitutes a representation of
$\mathfrak{sl} ( 2 )$ corresponding to the (not necessarily simple) roots $\pm \gamma_{j}$. In particular, we deduce that
$$e_{\gamma_{j}}^{m} u_{w\la} \neq 0 \hskip 5mm \text{and} \hskip 5mm  e_{\gamma_{j}}^{m+1} u_{w\la} = 0.$$
This implies that the map $\kappa$ is injective.

By constriction, the $N_{\gamma_{j}}$-cyclic $H$-eigenvector of $H^0(\bP^1,\cO(m-1))^*$ is annihilated by
$e_{\gamma_{j}}^{m}$ as the corresponding cyclic vector is annihilated by $e_{\gamma_{j}}^{m+1}$ in
$H^0(\bP^1,\cO(m))^*$. Sending it through $\kappa'$, the above commutative diagram asserts that
$e_{\gamma_{j}}^{m} v = 0$. This proves our Lemma.
\end{proof}

Recall that the star multiplication on $W$ is defined by $s_i*w=s_iw$ if $\ell(s_iw)=\ell(w)+1$ and $s_i*w=w$ otherwise ($i \in I$). This makes $(W, *)$ into a monoid. For each $k=1,\dots,l$, let $w[k]=s_{i_1}*\dots *s_{i_{k-1}} *s_{i_{k+1}}*\dots*s_{i_l}$.

\begin{lem}
Consider the embedding $$\varphi_k:H^0(\Q(\bi),\cO(\la-H_k))\hookrightarrow H^0(\Q(\bi),\cO(\la)).$$
Then ${\rm ker}\,\varphi_k^*\subset H^0(\Q(\bi),\cO(\la))^*\simeq {\mathbb W}_{w\la}$ is equal to
${\mathbb W}_{w[k]\la}$.
\end{lem}
\begin{proof}
For any $k=1,\dots,l$, we have the following exact sequence of sheaves:
\[
0\to\cO_{\Q(\bi)}(-H_k)\to \cO_{\Q(\bi)}\to \cO_{\Q(\bi')}\to 0,
\]
where $\bi'$ corresponds to omitting $i_k$ in $\bi = \{ i_{j} \}_{j}$. Note that
$$H^{i}(\Q(\bi'),\cO_{\Q(\bi')}(\la)) = H^{i}(\Q(\bi),\cO_{\Q(\bi')}(\la))$$
for each $i \ge 0$ as $\Q(\bi') \subset \Q(\bi)$ is a closed (ind-)subvariety. In view of \cite[Proposition 6.4]{Kat}, we apply $H^{0}(\cdot)^*$ to obtain
\[
0\to H^{0}(\Q(\bi'),\cO_{\Q(\bi')}(\la))^{*}\to H^{0}(\Q(\bi),\cO_{\Q(\bi)}(\la))^{*}\to
H^{0}(\Q(\bi),\cO_{\Q(\bi)}(\la-H_k))^{*}\to 0.
\]
By \cite[Lemma 6.1]{Kat} and the fact that the multiplication map of $\Q(\bi') \subset \Q(\bi)$ lands
exactly on $\Q(w[k])$, we deduce
\[
H^{0}(\Q(\bi'),\cO_{\Q(\bi')}(\la))^{*} \cong H^{0}(\Q(w[k]),\cO_{\Q(w[k])}(\la))^{*} \cong {\mathbb
W}_{w[k]\la}
\]
as required.
\end{proof}

\begin{thm}\label{g-real}
The $\I$-modules surjection ${\mathbb U}_{w\la}\to \Gamma(\Q(w),\E_w(\la))^*$ is an isomorphism.
\end{thm}

\begin{proof}
We have the following equality, where all the spaces are considered as subspaces of $H^0(\Q(\bi),\cO(\la))$:
\[
H^0(\Q(\bi),\cO(\la-\sum_{k=1}^l H_k))=\bigcap_{k=1}^l H^0(\Q(\bi),\cO(\la-H_k)).
\]
We conclude that
\[
H^0(\Q(\bi),\cO(\la-\sum_{k=1}^l H_k))^*\simeq \frac{H^0(\Q(\bi),\cO(\la))^*}{\sum_{k=1}^l {\rm
ker}\varphi_k^*}.
\]
So, theorem follows from the equality
\begin{equation}\label{w[k]}
\sum_{k=1}^l{\mathbb W}_{w[k] \la}=\sum_{\al\in \triangle_+\cap \sigma^{-1}\triangle_+} {\mathbb
W}_{s_{\sigma (\al)} w\la},
\end{equation}
where all the spaces are considered as subspaces of ${\mathbb W}_{w\la}$ (that is isomorphic to
$H^0(\Q(\bi),\cO(\la))^{*}$). Indeed, one has
\[
{\mathbb U}_{w\la}=\frac{{\mathbb W}_{w\la}}{\sum_{\al\in \triangle_+\cap \sigma^{-1}\triangle_+} {\mathbb
W}_{s_{\sigma (\al)} w\la}}.
\]
We prove \eqref{w[k]} in a separate lemma below.
\end{proof}

\begin{lem}\label{max-repr}
Let $\la$ be a dominant weight. Then for any element $w=\sigma w_0\in W$ so that $\ell(w)=l$,
we have the following equality of the subspaces of ${\mathbb W}_{w\la}$:
\[
\sum_{k=1}^l{\mathbb W}_{w[k]\la}=
\sum_{\al\in \triangle_+\cap \sigma^{-1}\triangle_+} {\mathbb W}_{s_{\sigma(\al)}w\la},
\]
\end{lem}
\begin{proof}
We first rewrite the left hand side. We take the maximal modules among the summands and conclude that the left
hand side is equal to
the sum over such $k=1,\dots,l$ such that $\ell(w[k])=l-1$ (i.e. after removing the $k$-th factor in the
reduced decomposition of $w$ we
still obtain a reduced expression). This is equivalent to saying that the left hand side is equal to the sum of
the global generalized Weyl
modules $\W_{s_\gamma w\la}$ such that there exists an edge $w^{-1}s_\gamma\to w^{-1}$ in the classical Bruhat
graph.

Now let us consider the right hand side. Taking the maximal summands, we only consider $\al$ such that there
exists an edge
$\sigma\to \sigma s_\al$ in the classical Bruhat graph (and we still have $\sigma (\al) \in\Delta_+$). Now let
$\gamma=\sigma(\al)$. Then
the right hand side is equal to the sum of the global generalized Weyl modules $\W_{s_\gamma w\la}$ such that
there is an edge
(recall $w=\sigma w_0$) from $ww_0$ to $s_\gamma ww_0$. Now taking inverse elements and multiplying by $w_0$, we
obtain that the
summands correspond to the edges  $w^{-1}s_\gamma\to w^{-1}$ in the classical Bruhat graph. This proves the
lemma.
\end{proof}

\section{Duality of local and global modules}\label{cat}
Throughout this section, we assume that $\g$ is of type $ADE$. In particular, $W^a$ is the affine Weyl group of $\g$. We extend the integral weight lattice $P$ of $\g$ to a weight lattice $P^{a}$ of the untwited affine Kac-Moody algebra $\gh$ corresponding to the simple Lie algebra $\g$ as
$$P^{a} := P \oplus \bZ \Lambda _{0} \oplus \bZ \delta,$$
where we regard $P \oplus \bZ \delta$ as the set of level zero integral weights, and $\Lambda_{0}$ is the level one basic fundamental weight, and $\delta$ is the primitive null-root. We denote by $\wh$ the Cartan subalgebra of $\gh$, and denote by $\al_{0}$ the affine simple root of $\gh$ (with its coroot $\al_{0}^{\vee}$). Let $s_{0} \in W^{a}$ be the simple reflection corresponding to $\al_{0}$. We set $I^{a} := \{0,1,\ldots,n\}$. We have
$$\{\al_{i}\}_{i \in I} \subset \{\al_{i}\}_{i \in I^{a}} \subset P^{a} \subset \wh^{*}.$$

We have a reduced expression
\begin{equation}
u(\la)= s_{i_1} s_{i_2} \dots s_{i_\ell} \pi,
\end{equation}
where $\pi$ is a length zero element in the affine Weyl group. Let $\Lambda := \pi \Lambda_{0}$. We have a level
one integrable highest weight representation $L ( \Lambda )$ associated to $\Lambda$, and $u({\la})$ defines a
Demazure submodule $D _{\la}$ of $L ( \Lambda )$ corresponding to $u({\la})\pi^{-1} \in W^{a}$. By its definition, $D_{\la}$ is a finite-dimensional $\wh$-semisimple $\I$-module. Moreover, it has a cyclic vector of weight $\la + \Lambda_{0}$ by our type $ADE$ assumption.

In addition, we regard the module $\mathbb U_{- \la}$ as a module whose cyclic vector has weight $- \la -
\Lambda_{0}$ (that is possible as the defining equation as $U ( \mathfrak n^{af} )$-modules completely
determines the structure of $\mathbb U_{-\la}$ up to $\wh$-weight twists; see Remark
\ref{rem:cartan}).

\begin{thm}[Sanderson-Ion \cite{S,I}]\label{SI}
We have $\ch\,D_{\la}=E_{\la}(x,q,0)$.
\end{thm}

Let $\fB$ be the category of $U ( \I )$-modules $M$ such that $M$ is semi-simple with respect to the
$\wh$-action with each $\wh$-weight space is at most countable dimension, and its weights belong to $P^a$. In particular, every module $M$ in $\fB$ admits a $\bZ$-grading coming from the $\bZ \delta$-part of the weight lattice (corresponding to the eigenvalues of the grading operator $d \in \wh \subset \gh$). In particular, $\fB$ is a graded abelian category.

Let $\fB'$ be the fullsubcategory of $\fB$ so that each $\wh$-weight space is finite dimensional, and its weights belong to $\Lambda + \sum_{i \in I ^a}\bZ_{\ge 0} \al_i$ for some $\Lambda \in P^a$. Let $\fB_{0} \subset \fB$ be the fullsubcategory consisting of finite-dimensional modules in $\fB$. The both $\fB'$ and $\fB_0$ are graded abelian categories.

\begin{lem}
The category $\mathfrak B$ has enough projectives.
\end{lem}

\begin{proof}
By $\wh$-semisimplicity, the maximal cyclic $\I$-module in $\mathfrak B$ that surjects onto $\bC_{\Lambda}$ is $U ( \I ) \otimes_{U (\h)} \bC _{\Lambda}$. By the Frobenius reciprocity, this module maps to every module in $\mathfrak B$ that has non-zero weight $\Lambda$-part. Collecting them for all weights, we obtain a surjection from a projective module to an arbitrary module in $\mathfrak B$ as required.
\end{proof}

Let $\Lambda \in P^{a}$. We denote by $\bC_{\Lambda}$ the one-dimensional $\I$-module whose action factors through
$$( \I + \wh ) \to ( \I + \wh ) / [\I,\I] \cong \wh \stackrel{\Lambda}{\longrightarrow} \bC.$$
Since $[\I,\I]$ is a (pro-)nilpotent Lie algebra, it follows that $\{ \bC_{\Lambda} \}_{\Lambda \in P^{a}}$ is the complete collection of simple modules in $\fB$. Let $P_{\Lambda}$ be the projective cover of $\bC_{\Lambda}$ in $\fB$.

\begin{prop}[see e.g. Kumar \cite{Kum} Chapter I\!I\!I]
For each $\Lambda \in P^{a}$, we have
\begin{equation}\label{projectivecharacter}
{\rm ch} \, P_{\Lambda}=\prod_{\alpha \in \Delta_+^{a}}(1-e^{\alpha})^{-{\rm mult}\, \al} \cdot {\rm ch} \, \bC_{\Lambda}.
\end{equation}
\end{prop}

\begin{proof}
Projective modules in $\fB$ are isomorphic to $U ( \I ) \otimes_{U (\gh)} \bC_{\Lambda}$. Hence the Poincar\'e-Birkoff-Witt theorem applied to $U ( [\I, \I ] )$ gives its character.
\end{proof}

For each $i \in I^{a}$ and a $U ( \I )$-module $M$, we define $\mathcal D_{i} ( M )$ to be the maximal $\mathfrak{sl} ( 2, i )$-integrable quotient of $U ( \I _{i} ) \otimes_{U ( \I )} M$.

\begin{lem}
For each $i \in I^{a}$, the functor $\mathcal D_{i}$ preserves $\fB$.
\end{lem}

\begin{proof}
For $M \in \fB$, the $U (\I)$-module
$$U ( \mathfrak{sl} ( 2, i ) + \I ) \otimes _{U ( \I )} M \cong U ( \mathfrak{sl} ( 2, i ) ) \otimes _{U ( \I \cap \mathfrak{sl} ( 2, i ) )} M$$
sits in $\fB$. Therefore, its quotient $\mathcal D_i M$ also lie in $\fB$ as required.
\end{proof}

\begin{thm}[Joseph \cite{J}]\label{Jos}
The functors $\{ \mathcal D_{i} \}_{i \in I^{a}}$ satisfy:
\begin{itemize}
\item Each $\mathcal D_{i}$ is right exact;
\item We have a natural transformation $\mathrm{Id} \to \mathcal D_{i}$;
\item For two $i,j \in I$ so that $(s_{i}s_{j})^{m} = 1$, we have
$$\overbrace{\mathcal D_{i} \mathcal D_{j} \cdots}^{m} \cong \overbrace{\mathcal D_{j} \mathcal
D_{i} \cdots}^{m};$$
\item We have $\mathcal D_{i}^{2} \cong \mathcal D_{i}$.
\end{itemize}
\end{thm}

\begin{proof}
We warn that Joseph's original formulation is for semi-simple Lie algebra, but the identical proof works for
Kac-Moody algebras. The first two assertions are \cite[Lemma 2.2]{J}. The third assertion is \cite[Proposition
2.15]{J}. Note that the functorial isomorphism in the third assertion follows from the fact that the resulting
functor yields a direct sum of finite-dimensional representations of simple Lie algebra generated by
$e_{\al_{i}},e_{\al_{j}}, f_{\al_{i}}, f_{\al_{j}}$. The fourth assertion follows as $\mathcal D_{i}$ does not change a
module that is $\mathfrak{sl} ( 2, \al_{i} )$-integrable.
\end{proof}

Let $D^{-} ( \fB )$ be the derived category of $\fB$ bounded from the below.
The restricted dual $*$ induces an endo-functor on the fullsubcategory of $\fB$-modules whose weight spaces are
finite-dimensional. It perserves $\mathfrak B_{0}$. We set $\mathcal D_{i}^{\dag} := * \circ \mathcal D_{i} \circ *$ for each $i \in I^{a}$. Let $\mathbb L \mathcal D _{i}$ (resp. $\mathbb R \mathcal D_{i}^{\dag}$) be the left derived
functor of $\mathcal D_{i}$ (resp. the $*$-conjugation of $\mathbb L \mathcal D _{i}$). The functor $\mathbb R \mathcal D_{i}^{\dag}$ lands on $\fB_0$ thanks to the following:

\begin{lem}\label{real}
Let $i \in I^{a}$. For each $N \in \mathfrak B_{0}$ and $k \in \bZ$, we have
$$H^{k} ( \mathbb R \mathcal D_{i}^{\dag} ( N ) ) \cong H^{k} ( \bP^{1}, \mathcal O ( N ) ),$$
where $\mathcal O ( N )$ is the $\mathop{SL} ( 2, \bC )$-equivariant vector bundle on $\bP^{1}$ obtained from
$N$. In particular, the total cohomology of $\mathcal D_{i}^{\dag} N$ lies in $\mathfrak B_{0}$.
\end{lem}

\begin{proof}
We have a functorial isomorphism (with respect to $N \in \mathfrak B_{0}$)
\begin{equation}
\left( U ( \mathfrak{sl} ( 2, i ) ) \otimes_{U ( \I \cap \mathfrak{sl} ( 2, i ) )} N^* \right)^{*} \cong
H^{0} ( U, \mathcal O ( N ) ) \cong \bC [ \bA ] \otimes N,\label{Visom}
\end{equation}
where $\bA \subset \bP^{1}$ denotes the open dense $\bfI$-orbit of $\bP^{1} \cong \bfI _{i} / \bfI$. The
maximal $\mathfrak{sl} ( 2, i )$-finite submodule of the LHS of (\ref{Visom}) is $H^{0} ( \mathbb R \mathcal
D_{i}^{\dag} ( N ) )$, and the maximal $\mathfrak{sl} ( 2, i )$-finite submodule of the RHS of (\ref{Visom})
is $H^{0} ( \bP^{1}, \mathcal O ( N ) )$. By construction, the both of $\{H^{k} ( \mathbb R \mathcal
D_{i}^{\dag} ( \bullet ) )\}_{k}$ and $\{H^{k} ( \bP^{1}, \mathcal O ( \bullet ) )^*\}_{k}$ are the universal
$\delta$-functors (as it is straight-forward to check that some finite-dimensional submodule of the injective
envelope yields an effacable envelope, see \cite[\S 2.1--2.2]{G}). Being universal $\delta$-functors of two
isomorphic functors, they are necessarily isomorphic as desired.
\end{proof}

\begin{prop}\label{adj}
Let $i \in I^{a}$. For $M \in \mathfrak B$ and $N \in \mathfrak B_{0}$, we have
$$\mathrm{ext}_{\mathfrak B}^{k} ( \mathbb L \mathcal D_{i} ( M ), N ) \cong \mathrm{ext}_{\mathfrak B}^{k} (
M, \mathbb R \mathcal D_{i}^{\dag} ( N ) ) \hskip 5mm k \in \bZ,$$
where $\mathrm{ext}$ are understood as the hypercohomologies.
\end{prop}

\begin{proof}
We set $\al := \al_{i}$. Let us denote by $\fb _0 := \wh \oplus E_\al = ( \I \cap ( \wh + \mathfrak{sl} ( 2, i ) ) )$ and $\g_0 := \wh
+ \mathfrak{sl} ( 2, i )$. Let us denote by $V_0 (\Lambda)$ be the irreducible finite-dimensional $( \g_0 +
\wh )$-module with highest weight $\Lambda \in P^a$. For each $\Lambda \in P^{a}$, we have
$$\mathcal D_{i} ( U ( \I ) \otimes_{U (\wh)} \mathbb C_{\Lambda} ) \cong \begin{cases} \bigoplus_{n \ge 0} U
( \I ) \otimes_{U (\fb_0)} V_0 ( \Lambda + n \alpha ) & (\bra \al^{\vee}, \Lambda \ket \ge 0)\\\bigoplus_{n \ge
0} U ( \I ) \otimes_{U (\fb_0)} V_0 ( s_{i} \Lambda + n \alpha ) & (\bra \al^{\vee}, \Lambda \ket <
0)\end{cases}.$$
Let us consider a $\wh$-semisimple indecomposable $\fb_0$-module $N_{\Gamma, m}$ with lowest weight $\Gamma \in
P^{a}$ and highest weight $\Gamma + m \al$ (such a module is unique up to isomorphism, cf. \cite[\S 2.3]{J}). We
have
$$\mathcal D_{i} ^{\dag} ( N_{\Gamma, m} ) \cong \begin{cases}\bigoplus _{0 \le n \le \min \{m, -m - \bra \al^{\vee}, \Gamma
\ket\}} V_0 ( s_{i} \Gamma - n \al ) & (m \le - \bra \al^{\vee}, \Gamma \ket)\\ \{ 0 \} & (m > - \bra \al^{\vee}, \Gamma
\ket)\end{cases}.$$
By the Frobenius-Nakayama reciprocity, we have
\begin{equation}
\mathrm{ext}_{\fB}^{k} ( U ( \I ) \otimes_{U (\wh)} M, N ) \cong \mathrm{ext}_{(\fb_0, \wh)}^{k} ( U (\fb_0)
\otimes_{U(\wh)} M, N ),\label{FN-rec0}
\end{equation}
for a finitely generated $U ( \fb_0 )$-module $M$ with semi-simple $\wh$-action and $N \in \fB_{0}$, where the RHS denotes the relative extension (cf. Kumar \cite[Chapter I\!I\!I]{Kum}).

In view of \cite[\S 2.3]{J}, it suffices to compute the extension by replacing $U ( \I ) \otimes_{U (\wh)} M$ with $\mathbb L \mathcal D_{i} ( U ( \fb_0 ) \otimes_{U (\wh)} \mathbb C_{\Lambda} )$ and $N$ with a string $U ( \fb_0 )$-module to see the desired isomorphism for a projective module $M$ and a finite-dimensional module $N$.

In other words, our assertion reduces to the functorial isomorphism:
\begin{equation}
\mathrm{ext}_{(\fb_0, \wh)}^{k} ( \mathbb L \mathcal D_{i} ( U ( \fb_0 ) \otimes_{U ( \wh )} \bC_{\Lambda} ),
N_{\Gamma, m} ) \cong \mathrm{ext}_{(\fb, \fh)}^{k} ( U (\fb_0) \otimes_{U(\wh)} \bC_{\Lambda}, \mathbb R \mathcal D_{i} ^{\dag} ( N_{\Gamma, m} ) ),\label{FN-rec}
\end{equation}
for $\Lambda, \Gamma \in P^{a}$, $m \in \bZ_{\ge 0}$, and $k \in \bZ$, where $\mathcal D_{i}$ and $\mathcal D_{i}^{\dag}$ are replaced with analogous functors to $\mathcal D_{i}$ and $\mathcal D_{i}^{\dag}$ defined
for $\fb_0$-modules with semi-simple $\wh$-actions. Our functors in (\ref{FN-rec}) are universal
$\delta$-functors as $\mathbb L^\bullet\mathcal D_{i}$ is coeffaceable by taking projective cover, and
$\mathbb R^\bullet \mathcal D_{i}^{\dag}$ is effaceble by taking a finite-dimensional submodule inside its
injective envelope (cf. \cite[\S 2.1--2.2]{G}). Therefore, it suffices to prove (\ref{FN-rec}) for $k=0$.

Here the $k = 0$ case of the RHS of (\ref{FN-rec}) further reduces to
\begin{align*}
\mathrm{hom}_{(\fb_0, \wh)} \, & ( U (\fb_0) \otimes_{U(\wh)} \bC_{\Lambda}, \mathcal D_{i} ^{\dag} (
N_{\Gamma, m} ) ) \cong \mathrm{hom}_{\wh} ( \bC_{\Lambda}, \mathcal D_{i} ^{\dag} ( N_{\Gamma, m} ) )\\
& \cong \begin{cases}\bigoplus _{0 \le n \le n_0} \mathrm{hom}_{(\fb_0, \wh)} ( \bC_{\Lambda}, V_0 ( s_{i} \Gamma - n \al ) ) & (m \le - \bra \al^{\vee}, \Gamma \ket)\\ \{ 0 \} & (m > - \bra \al^{\vee}, \Gamma \ket)\end{cases}
\end{align*}
by the Frobenius reciprocity, where $n_0 := \min \{m, -m - \bra \al^{\vee}, \Gamma \ket\}$. The $k = 0$ case of
LHS of (\ref{FN-rec}) is rephrased as:
\begin{align*}
\mathrm{hom}_{(\fb_0, \wh)} \, & ( \mathcal D_{\al} ( U ( \fb_0 ) \otimes_{U ( \wh )} \bC_{\Lambda} ),
N_{\Gamma, m} ) \\
& \cong \begin{cases} \bigoplus_{n \ge 0} \mathrm{hom}_{(\fb_0, \wh)} ( V_0 ( \Lambda + n \alpha ), N_{\Gamma,
m} ) & (\bra \al^{\vee}, \Lambda \ket \ge 0)\\\bigoplus_{n \ge 0} \mathrm{hom}_{(\fb_0, \wh)} ( V_0 ( s_{i} \Lambda + n \alpha ), N_{\Gamma, m} ) & (\bra \al^{\vee}, \Lambda \ket < 0)\end{cases}
\end{align*}
From these, we derive the desired isomorphisms (\ref{FN-rec}). Moreover, these isomorphisms are functorial with respect to the morphism of modules as it commutes with the morphisms in each variable.

Therefore, we conclude the desired functorial isomorphism as required.
\end{proof}

Below in this section, every functor is derived unless stated otherwise.

\begin{lem}[\cite{Kat} \S 4, particularly Theorem 4.13]\label{DCF}
For each $i \in I$ and $\la \in P$, we have
$$\mathcal D_i ( \mathbb W _{\la} ) \cong \begin{cases}\mathbb W_{s_i \la} & (\bra \al^{\vee}_{i}, \la \ket >0)
\\ \mathbb W_{\la} & (\bra \al^{\vee}_{i}, \la \ket \le 0)  \end{cases}.$$
\end{lem}

\begin{prop}\label{ind}
For each weight $\lambda$ and $i \in I^{a}$, we have
\begin{align}
H_{k} ( \mathcal D_{i} ( D_{\la} ) ) & \cong \begin{cases}D_{s_{i} \la} & (k = 0, u({\la}) < s_{i} u({\la})
\not\in u(\la)W)\\ D_{\la} & (k = 0, u_{\la} > s_{i} u(\la) \text{ or } s_{i} u(\la) \in u(\la) W)\\
\{0\} & (k \neq 0)\end{cases}
\end{align}
Assume that $\g$ is not of type $E_{8},F_{4},$ or $G_2$. If we have $s_{i} u(\la) < u(\la)$, then we have a short exact sequence
$$0 \to \mathbb U_{-\la} \to \mathcal D_{i} ( \mathbb U_{-\la} ) \to \mathbb U_{- s_{i} \la} \to 0.$$
If we have $s_{i} u(\la) \in u(\la) W$, then we have $\mathbb U_{-\la} \cong \mathcal D_{i} ( \mathbb U _{- \la} )$.
If we have $u(\la) < s_{i} u(\la) \not\in u(\la) W$, then we have $\mathcal D_{i} ( \mathbb U _{- \la} ) \cong \{ 0 \}$. In each case, we have $H_k ( \mathcal D_{i} ( \mathbb U_{-\la} ) ) \cong \{ 0 \}$ for $k \neq 0$.
\end{prop}

\begin{proof}
The first assertion is a rephrasement of \cite[Theorem 8.2.2 and Theorem 8.2.9]{Kum} applied to $\pi
\Lambda_{0} \in P^{a}$ and $u(\la) \pi^{-1} \in W^{a}$.

We prove the second assertion. In view of the construction of \cite[\S 6]{Kat} (see also \S \ref{SoSIS}) and
Lemma \ref{real}, we deduce a short exact sequence
\begin{equation}
0 \to \mathbb U_{- \la} \to \mathcal D_{i} ( \mathbb U_{- \la} ) \to \mathbb U_{- s_{i} \la} \to
0\label{U-ses}
\end{equation}
when $s_i u(\la) < u(\la)$. Moreover, the corresponding higher cohomologies must vanish. By Theorem \ref{Jos}
4), it holds that applying $\mathcal D_{i}$ to (\ref{U-ses}) yields an isomorphism $\mathcal D_{i} ( \mathbb
U_{- \la} ) \to \mathcal D_{i}^{2} ( \mathbb U_{- \la} )$. Hence, we have $\mathcal D_{i} ( \mathbb U_{-
s_{i}\la} ) \cong \{ 0 \}$ by the associated long exact sequence.

We consider the case $s_i u(\la) \in u(\la) W$. We have $\la = w \la_-$ for $w \in W$. In view of the last
formula in the proof of Theorem \ref{g-real} and Lemma \ref{max-repr}, we know that $\mathbb U_{- w \la_-}$ is a
quotient of $\mathbb W _{- w \la_-}$ by $\mathbb W_{-u\la_-}$ with all $u < w \in W$. Here, we have $s_\al \la =
\la$, which implies that $\mathbb W _{- w \la_-} = \mathbb W _{- s_i w \la_-}$ and $s_i w\cdot \mathrm{stab}_W
\la_- = w\cdot \mathrm{stab}_W \la_-$. Note that $w$ can be thought of as a minimal length element in $w\cdot
\mathrm{stab}_W \la_- \subset W$. It follows that $s_i u \not\in w\cdot \mathrm{stab}_W \la_-$ and $s_i u < s_i
w$. It implies
$$\mathbb W _{- s_i u \la_-} \subsetneq \mathbb W _{- s_i w \la_-} = \mathbb W _{- w \la_-}.$$
Therefore, we have necessarily $\mathbb W _{- v \la_-} = \mathbb W _{- s_i u \la_-}$ with $v < w$ in view of
\cite[Theorem 4.12 (2)]{Kat}. This implies $\mathcal D_i ( \mathbb U _{-\la} ) \cong \mathbb U_{-\la}$ by
Lemma \ref{DCF} and the last formula in the proof of Theorem \ref{g-real}.

In case $i = 0$, then we apply a diagram automorphism $\tau$ of the affine Dynkin diagrams of type
$ABCDE_{6}E_{7}$ to $\alpha_{0}$ and $\mathbb U_{- \la}$. Then, $\tau \alpha_{0} = \al_{i}$
for some $i \in I$, and $\tau (\la + \Lambda_{0}) = \la' + \Lambda_{0}$ for some $\la' \in P$ so that
$\left< \al^{\vee}_{i}, \la' \right> < 0$. In addition, we have $\tau ( \mathbb U_{- \la} ) \cong \mathbb U_{-
\la'}$ by the description of the defining equations. Therefore, we deduce the assertion also in this case.
\end{proof}

\begin{lem}\label{W-res}
Suppose that $\la$ is anti-dominant. The module $\mathbb W_{\la}$ admits a resolution by $\{ P_{\gamma} \left< m \right> \}_{\gamma, m \in \bZ}$, where $w \gamma < \la$ for some $w \in W$ or $\gamma = \la$ $($up to $\bZ \Lambda_0$-character twists$)$.
\end{lem}

\begin{proof}
By a result of Chari-Ion \cite{CI}, we deduce that $\mathbb W_{\la}$ admits a resolution by $U ( \g [z] )
\otimes _{U ( \g )} V ( w_0 \gamma )$ (that is a projective module in the category of $\g$-integrable $\g [z]$-modules, see \cite{CG}), where $\gamma \le \la$. Since $V ( w_0 \gamma )$ admits a finite
resolution by $\{U ( \fb ) \otimes _{U ( \h )} \mathbb C _{w ( \gamma - \rho ) + \rho} \}_{w \in W}$ (afforded by the BGG resolution) as $U ( \g )$-modules, we deduce that the total complex of the double complex resolving each $U ( \g [z] ) \otimes _{U ( \g )}
V ( w_0 \gamma )$ (direct summand of a term in the projective resolution in the category of $\g$-integrable $\g [z]$-modules) has $\{ P _{w ( \gamma - \rho ) + \rho} \}_{\gamma, w \in W}$ ($\gamma$ is as above) as its direct factors. Since we have
$$w^{-1} ( w ( \gamma - \rho ) + \rho ) \le \gamma,$$
and the equality holds if and only if $w = e$, we conclude the assertion.
\end{proof}

\begin{lem}\label{D-res}
Assume that $\g$ is of type $ADE$. Suppose that $\mu \in P_-$. The module $D_{\mu}$ admits a resolution by
$\{ P_{\gamma} \left< m \right> \}_{\gamma, m \in \bZ}$, where $\gamma \in P$ satisfies $w \gamma < \mu$
for some $w \in W$ or $\gamma = \mu$ $($up to $\bZ \Lambda_0$-character twists$)$.
\end{lem}

\begin{proof}
The module $W_{\mu}$ admits a finite resolution by a complex whose terms are the direct sum of $\mathbb W_{\mu}$ (since $\mathbb W_{\mu}$
admits an action of a polynomial ring and its specialization to a point is $W_\mu$ by
\cite{FL, N}). Hence, Lemma \ref{W-res} implies that $W_{\mu}$ admits a $U ( \I )$-module resolution of the
desired type. Therefore, the identification $D_{\mu} \cong W _{\mu}$ (see Remark \ref{WDU-rel}) implies the
result.
\end{proof}

\begin{thm}\label{orth}
Assume that $\g$ is of type $ADE_{6}E_{7}$. We have:
\[
\mathrm{ext}^{i}_{\mathfrak B} ( \mathbb U_{- \la}, D_{\mu}^{*} ) \cong \begin{cases} \bC & (i=0,\la=\mu)\\\{0\}
& (\text{otherwise}).\end{cases}
\]
\end{thm}

\begin{proof}
If $\la - \mu \not\in Q$, then the extension trivially vanish.

If we have $i \in I^{a}$ so that $s_{i}u(\mu) < u(\mu)$ or $s_{i}u_{\mu} \in u(\mu) W$ and $u(\la) < s_{i}
u(\la) \not\in u (\la) W$, then we have
\begin{align*}
\mathrm{ext}^{\bullet}_{\fB} ( \mathbb U_{- \la}, D_{\mu}^{*} ) & \cong \mathrm{ext}^{\bullet}_{\fB} ( \mathbb
U_{- \la}, \mathcal D_{i}^{\dag} ( D_{\mu}^{*} ) )\\
& \cong \mathrm{ext}^{\bullet}_{\fB} ( \mathcal D_{i}( \mathbb U_{- \la} ), D_{\mu}^{*} )\\
& \cong \mathrm{ext}^{\bullet}_{\fB} ( \{0\}, D_{\mu}^{*} ) = \{0\}.
\end{align*}
This particularly implies
$$\mathrm{ext}^{\bullet}_{\fB} ( \mathbb U_{- \la}, D_{\mu}^{*} ) = \{ 0 \}$$
whenever there exists $i \in I$ so that $\bra \al^{\vee}_{i}, \la \ket > 0 \ge \bra \al^{\vee}_{i}, \mu \ket$ or
$\bra \vartheta^{\vee}, \la \ket \le 0 < \bra \vartheta ^{\vee}, \mu \ket$.

Assume that $\la$ and $\mu$ are both anti-dominant. Applying Lemma \ref{D-res}, we obtain an injective
resolution of $D_{\mu}^{*}$ as $U ( \I )$-module whose simple submodules are $\mathbb C_{\gamma}$, where $w
\gamma < \mu$ for some $w \in W$ or $\gamma = \mu$. Therefore, we conclude that
\begin{equation}
\mathrm{ext}^{\bullet}_{\fB} ( \mathbb U_{- \la}, D_{\mu}^{*} ) = \{ 0 \} \hskip 5mm \la \not\le
\mu.\label{antidom-vanish1}
\end{equation}

Assume that $\la$ and $\mu$ are both anti-dominant. By Remark \ref{WDU-rel}, we have $\mathbb U_{-\la} = \mathbb
W_{-\la}$. Applying Lemma \ref{W-res}, we have
\begin{align}\nonumber
\mathrm{ext}^{\bullet}_{\fB} ( \mathbb U_{- \la}, D_{\mu}^{*} ) & \cong \mathrm{ext}^{\bullet}_{\fB} ( \mathbb
W_{- \la}, \mathcal D_{w_0}^{\dagger} ( D_{w_0 \mu}^{*} ) ) &\\\nonumber
& \cong \mathrm{ext}^{\bullet}_{\fB} ( \mathbb W_{- w_0 \la}, D_{w_0 \mu}^{*} ) & \\
& = \{ 0 \} & \la \not\ge \mu.\label{antidom-vanish2}
\end{align}

We calculate the $\mathrm{ext}$-groups when $\la = \mu = 0$. We have $D_0 \cong \bC_0$. Then, we can identify
the projective resolution of $\bC_{0}$ with the BGG resolution of $D_{0}$ in terms of the lowest weight Verma
modules of $\gh$. In particular, the head of a projective resolution of $D_{0}$ in $\fB$ has weight $- W^{a}
\rho^{a} + \rho$, where $\rho^{a}$ is an arbitrary weight in $P^{a}$ so that $\bra \al^{\vee}_{i},
\rho^{a} \ket = 1$ for every $i \in I^{a}$. In addition, each $w \in W^{a}$ corresponds to a single projective
module in the BGG resolution. Therefore, the $\wh$-eigen cyclic generators of weight $0$ appears only
once at the zero-th term. This implies
$$\mathrm{ext}^{k}_{\mathfrak B} ( \mathbb U_{0}, D_{0}^{*} ) \cong \begin{cases} \bC & (k=0)\\\{0\} &
(\text{otherwise})\end{cases}.$$

Summarizing the above, we have
$$\mathrm{ext}^{k}_{\fB} ( \mathbb U_{- \la}, D_{0}^{*} ) = \begin{cases} \bC & (k = 0, \la = 0)\\ \{0\} &
(otherwise)\end{cases}.$$

We prove the main assertion by induction. Namely, we prove
$$\mathrm{ext}^{k}_{\fB} ( \mathbb U_{- \la}, D_{\gamma}^{*} ) = \begin{cases} \bC & (k = 0, \la =
\gamma)\\ \{0\} & (otherwise)\end{cases}$$
for $\gamma \in P$ by assuming the same assertion for every $\mu \in \Lambda$ so that $u(\mu) <
u(\gamma)$. The initial case $\gamma = \tau \Lambda_0$ for $\tau \in \Pi$ follows by the previous paragraph by
applying a diagram automorphism of $\gh$ arising from $\tau$ (if $\tau \neq 1$). Hence, we can also assume
$\gamma \not\in \Pi \Lambda_0$ in addition.

We have some $i \in I^{a}$ so that $u(\mu) = s_{i} u(\gamma)$ and $u(\mu) < u(\gamma)$. Then, we have
\begin{align*}
\mathrm{ext}^{\bullet}_{\fB} ( \mathbb U_{- \la}, D_{\gamma}^{*} ) & \cong \mathrm{ext}^{\bullet}_{\fB} (
\mathbb U_{- \la}, \mathcal D_{i}^{\dag} ( D_{\mu}^{*} ) )\\
& \cong \mathrm{ext}^{\bullet}_{\fB} ( \mathcal D_{i} ( \mathbb U_{- \la} ), D_{\mu}^{*} ).
\end{align*}
In view of Proposition \ref{ind}, we have $\mathcal D_{i} ( \mathbb U_{- \la} ) \cong \{ 0\}$ if $\bra
\al^{\vee}_{i}, \la + \Lambda_0 \ket > 0$, and $\mathcal D_{i} ( \mathbb U_{- \la} ) \cong \mathbb U_{- \la}$ if
$\bra \al^{\vee}_{i}, \la + \Lambda_0 \ket = 0$. In these cases, we have $\la \neq \gamma$, and the induction
hypothesis yields
$$\mathrm{ext}^{\bullet}_{\fB} ( \mathbb U_{- \la}, D_{\gamma}^{*} ) = \{ 0 \}.$$
In case $\bra \al^{\vee}_{i}, \la + \Lambda_0 \ket < 0$, then we have (a part of) the long exact sequence
$$\to \mathrm{ext}^{\bullet}_{\fB} ( \mathbb U_{- s_i \la}, D_{\mu}^{*} ) \to \mathrm{ext}^{\bullet}_{\fB} (
\mathcal D_{i} ( \mathbb U_{- \la} ), D_{\mu}^{*} ) \to \mathrm{ext}^{\bullet}_{\fB} ( \mathbb U_{- \la},
D_{\mu}^{*} ) \to \mathrm{ext}^{\bullet+1}_{\fB} ( \mathbb U_{- s_i  \la}, D_{\mu}^{*} ).$$
As a consequence, we have non-zero result if and only if $s_i \la = \mu$ or $\la = \mu$ by the induction
hypothesis. The latter case is prohibited by the comparison of $\bra \al^{\vee}_{i}, \la + \Lambda_0 \ket < 0$
and $s_i u(\gamma) < u(\gamma)$. Therefore, we conclude
$$\mathrm{ext}^{\bullet}_{\fB} ( \mathbb U_{- \la}, D_{\gamma}^{*} ) \cong \mathrm{ext}^{\bullet}_{\fB} (
\mathbb U_{- s_i \la}, D_{s_i \gamma}^{*} ) = \mathrm{ext}^{\bullet}_{\fB} ( \mathbb U_{- s_i \la}, D_{\mu}^{*}
).$$
Therefore, our induction hypothesis proceeds the induction as required.
\end{proof}

\appendix
\section{Numerical equality}

We discuss the equality of Theorem \ref{orth} on the level of characters.

Consider the Cherednik kernel:
\[
\kappa (x,q,t)=\frac{\prod_{\alpha \in \Delta_+^{a}}(1-e^{\alpha})^{{\rm mult}\, \al}}
{\prod_{\alpha \in \Delta_+^{a}}(1-te^{\alpha})^{{\rm mult}\, \al}}
\in \mathbb{C}[P](q,t) = \mathbb{C}[x_{1}^{\pm 1}, \ldots, x_{n}^{\pm 1}](q,t),
\]
through the identifications $e^{\om_{i}} = x_{i}$ for $1 \le i \le n$ and $q = e^{\delta}$.

We consider the Euler-Poincar\'e pairing
\begin{equation}\label{extpair}
[\fB'] \times [\fB_{0}] \ni (M, N) \mapsto ( M, N )_{EP} := \sum_{i=0}^{\infty}(-1)^i \mathsf{gdim} \, \mathrm{ext}^i( M,N^*)^{*},
\end{equation}
as the formal sum. This pairing lands in $\bC (\!( q )\!)$.

The Euler-Poincar\'e pairing satisfies the following properties:
\begin{enumerate}
\item It is $q$-linear;
\item For a short exact sequence:
\[0 \rightarrow M_1 \rightarrow M \rightarrow M_2 \rightarrow 0,\]
we have
\[(M,N)_{EP}=(M_1,N)_{EP}+(M_2,N)_{EP}\]
and the same equality holds for a short exact sequence in the second argument. Thus the Euler-Poincar\'e pairing depends only on the characters of $M$ and $N$;
\item We have the following equality:
$$(P_\Lambda,\bC_\Gamma)_{EP}=\delta_{\Lambda,-\Gamma} \hskip 5mm \Lambda, \Gamma \in P^{a};$$
\item If both $M$ and $N$ belong to $\fB_{0}$, then we have
$$(M,N)_{EP}=(N,M)_{EP}.$$
\end{enumerate}

The proofs of these properties are standard and is omitted (the last item requires \cite[\S 2.1--2.2]{G} as in the previous section).

The properties $(\mathrm{i})$, $(\mathrm{ii})$, $(\mathrm{iii})$ completely characterizes the Euler-Poincar\'e pairing. Now consider the specialization of the Cherednik inner product on $\bC[x^{\pm 1}](\!(q)\!)$:

\begin{equation}
(P(x,q),Q(x,q))_{C} := (P(x,q)Q(X,q)\kappa(x,q,0))_0,
\end{equation}
where the lower index $0$ denotes the constant term with respect to $q$ in the power series expansion of $h$.
Applying $(\mathrm{i})$, $(\mathrm{ii})$, $(\mathrm{iii})$ repeatedly, we obtain:
\begin{equation}
(M,N)_{EP}=({\rm ch} \, M,{\rm ch} \, N )_{C}.\label{EP=C}
\end{equation}

\begin{thm}
For each $\la,\mu \in P$ so that $\la \neq \mu$, we have
\begin{equation}
( \mathrm{ch} \, U_\mu, \mathrm{ch} \, D_\la )_{C} = 0 = ( U_{\mu}, D_\la )_{EP}.
\end{equation}
\end{thm}

\begin{proof}
For $f (x,q,t) \in \bC [P](q,t)$, we set
$$\overline{f(x,q,t)}=f(x^{-1},q^{-1},t^{-1}), \hskip 3mm f^*(x,q,t)=f(x^{-1},q^{-1},t).$$

By the definition of the nonsymmetric Macdonald polynomials (see e.g. \cite{Ch1}), we have
\[\left(E_{\lambda}(x,q,t)\overline{E_{\mu}(x,q,t)}\kappa(x,q,t)\right)_0=0\]
for $\lambda \neq \mu$. In other words:
\[\left(E_{\lambda}(x,q,t)E^* _{\mu}(x,q,t^{-1})\kappa(x,q,t)\right)_0=0.\]

Substituting $t=0$, we obtain
\[\left(E_{\lambda}(x,q,0)( w_0 E_{w_0(\mu)}(x,q^{-1},\infty))\kappa(x,q,0)\right)_0=0.\]
In view of Corollary \ref{UE} and Theorem \ref{SI}, we conclude that the first equality. The second equality follows from \eqref{EP=C} and $(\mathrm{iv})$.
\end{proof}

\section*{Acknowledgments}
The research of E.F. and I.M. is supported by the grant RSF-DFG 16-41-01013.
The research of S.K is supported in part by JSPS Grant-in-Aid for Scientific Research (B) JP26287004.

\end{document}